\documentclass{article}
\usepackage[utf8]{inputenc}
\usepackage{fullpage}
\usepackage{url}
\usepackage{hyperref}

\usepackage{amsmath,amsthm,amssymb,enumerate}
\usepackage{mathtools}
\allowdisplaybreaks

\usepackage[capitalise]{cleveref}

\newtheorem{theorem}{Theorem}[section]
\newtheorem{proposition}[theorem]{Proposition}
\newtheorem{lemma}[theorem]{Lemma}

\newtheorem{remark}[theorem]{Remark}
\newtheorem{definition}[theorem]{Definition}
\newtheorem{assumption}[theorem]{Assumption}

\usepackage{enumitem}
\usepackage{xcolor}
\usepackage{xspace}  

\newcommand{\node}{vertex\xspace}
\newcommand{\nodes}{vertices\xspace}

\newcommand{\nbr}[1]{N_{#1}}
\newcommand{\state}[1]{x^{(#1)}}
\newcommand{\bstate}[1]{\mathbf{x}^{(#1)}}
\newcommand{\snorm}[1]{z^{(#1)}}

\newcommand{\Ex}{\mathbb{E}}
\renewcommand{\Pr}{\mathbb{P}}

\title{On the convergence of nonlinear averaging dynamics \\with three-body interactions on hypergraphs}
\author{
    Emilio Cruciani\\
    {\small University of Salzburg}\\
    {\small \texttt{emilio.cruciani@plus.ac.at}}
    \and
    Emanuela L. Giacomelli\\
    {\small LMU Munich}\\
    {\small \texttt{emanuela.giacomelli@math.lmu.de}}
    \and
    Jinyeop Lee\\
    {\small LMU Munich}\\
    {\small \texttt{lee@math.lmu.de}}
}
\date{}

\begin{document}

\maketitle

\begin{abstract}
Complex networked systems in fields such as physics, biology, and social sciences often involve interactions that extend beyond simple pairwise ones.
Hypergraphs serve as powerful modeling tools for describing and analyzing the intricate behaviors of systems with multi-body interactions.
Herein, we investigate discrete-time dynamics with three-body interactions, described by an underlying 3-uniform hypergraph, where \nodes update their states through a nonlinearly-weighted average depending on their neighboring pairs' states.
These dynamics capture reinforcing group effects, such as peer pressure, and exhibit higher-order dynamical effects resulting from a complex interplay between initial states, hypergraph topology, and nonlinearity of the update.
Differently from linear averaging dynamics on graphs with two-body interactions, this model does not converge to the average of the initial states but rather induces a shift.
By assuming random initial states and by making some regularity and density assumptions on the hypergraph, we prove that the dynamics converge to a multiplicatively-shifted average of the initial states, with high probability. 
We further characterize the shift as a function of two parameters describing the initial state and interaction strength, as well as the convergence time as a function of the hypergraph structure.

\end{abstract}

\vfill 
\subsubsection*{Acknowledgments}
{\small The authors acknowledge Li Chen for a helpful discussion about nonlinear parabolic partial differential equations.}

\subsubsection*{Funding}
{\small Department of Computer Science, University of Salzburg. This work is supported by the Austrian Science Fund (FWF): P 32863-N. This project has received funding from the European Research Council (ERC) under the European Union’s Horizon 2020 research and innovation programme (grant agreement No 947702).}
\newpage

\section{Introduction}

Complex networks are powerful tools to model a variety of phenomena in natural and social sciences, ranging from biological systems to social networks. 
A common approach to study the structure and dynamics of complex networks is to assume that the interactions among its entities are pairwise, thus describing them by the edges of an underlying graph~\cite{BOCCALETTI2006175}.
However, this assumption may neglect important dynamical effects that arise from \textit{higher-order interactions}, i.e., interactions that involve more than two \nodes at a time. 
Higher-order interactions consider the group effect as a whole so that it cannot be derived by a combination of pairwise interactions.
Diverse real-world systems cannot accurately be described by pairwise interactions given that fundamental interplay takes place at a collective level, for example in ecology~\cite{billick1994higher}, particle physics~\cite{RevModPhys.85.197}, neuroscience~\cite{schneidman2006weak}, and social network science~\cite{sekara2016fundamental,patania2017shape}.

A natural way of modeling higher-order interactions is that of representing them using simplicial complexes~\cite{BATTISTON20201}, which glue together simplices such as points, segments, triangles, tetrahedrons, and their higher-dimensional counterparts. Alternatively, they can be modeled through the use of hypergraphs, generalizations of graphs where hyperedges can connect any number of \nodes.
Several types of dynamical processes have been recently studied on these higher-order interaction models, such as social contagion~\cite{iacopini2019simplicial}, opinion dynamics~\cite{PhysRevE.104.024316,PhysRevE.101.022305}, oscillator synchronization~\cite{gambuzza2021stability}, and random walks~\cite{PhysRevE.101.022308}.
Multi-body interaction systems employing simplicial complexes can take advantage of powerful algebraic structures. However, the resulting processes are defined between simplices of varying dimensions, rather than solely between \nodes as is typically the case in (hyper)graphs.
The modeling using simplicial complexes gives rise to challenges in both the analysis and the interpretability of the models~\cite{BATTISTON20201}.

Opinion formation and social learning processes have been the focus of extensive research in the area of complex networks. 
One of the first, and arguably simplest, models for these kinds of phenomena was introduced by DeGroot~\cite{degroot1974reaching}.
The original process works as follows: it is given a graph of $n$ interconnected agents (possibly edge-weighted and directed); each agent holds an initial opinion on some subject (a scalar value); 
in each discrete-time step, the agents update their opinions by performing an average of the opinions of their neighbors.
It is proven that, for connected and non-bipartite graphs, such a process eventually reaches a \textit{consensus}, namely the agents agree on the same opinion, which turns out to be the degree-weighted average of the initial opinions of the agents.
DeGroot's model has been widely investigated in physics, computer science, and more in general in the theory of social networks.
It has been extended in several ways, most notably by Friedkin and Johnsen~\cite{friedkin1990social} which assume that agents have some immutable internal opinion that affects the averaging process.
A combination of the two models has also been considered to describe self-appraisal, social power, and interpersonal influences~\cite{doi:10.1137/130913250}.
We point the reader to a tutorial on the modeling and analysis of opinion dynamics on social networks~\cite{PROSKURNIKOV201765,PROSKURNIKOV2018166}.
We remark that the averaging process is the other side of the coin of a diffusion (or random walk) process on the same graph: in the former, \nodes ``pull'' information from their neighbors; in the latter, a mass on the \nodes is ``pushed'' to their neighbors.
This matter is discussed in more detail in \cref{sec: heat}.

Continuous-time models of opinion formation with higher-order interactions have been recently developed over simplicial complex structures~\cite{deville2021consensus} and hypergraphs~\cite{doi:10.1137/21M1399427,PhysRevE.101.032310,Sahasrabuddhe_2021}. 
As discussed earlier, we focus on hypergraphs for their simplicity and interpretability.
In particular, three-body dynamical systems are considered in~\cite{PhysRevE.101.032310}, namely on hypergraphs comprising triples of \nodes.
In the model, each \node updates its state as a function of the combined values of all three \nodes in each hyperedge.
That paper proves that linear dynamics on hypergraphs can be rewritten as standard dynamics with two-body interaction on a graph, namely coinciding with DeGroot's model.
Such observation puts emphasis on the fact that nonlinear functions are necessary to capture higher-order effects induced by multi-body interactions.
In the same article, specific nonlinear functions have been considered to describe consensus models with reinforcing group effects such as peer pressure and stubbornness. 
Unlike DeGroot's model, it is experimentally shown that the resulting dynamics can cause shifts away from the average system state. 
The nature of these shifts depends on a complex interplay between the distribution of initial states, underlying structure, and form of interaction function.

Such a framework has been extended to include multi-body interactions on general hypergraphs, i.e., among heterogeneous groups of \nodes instead of triples~\cite{Sahasrabuddhe_2021}.
Further extensions of the model investigate consensus dynamics on temporal hypergraphs that encode network systems with time-dependent, multi-way interactions~\cite{PhysRevE.104.064305}.
We remark that these processes are also strictly linked to graph neural network dynamics. Such a relationship will be discussed more in detail in \cref{sec: gnns}.

The dynamics with multi-body interactions on hypergraphs we have previously discussed have mostly been analyzed experimentally, with no analytical results on their convergence.
In fact, to the best of our knowledge, there are only few rigorous analytical results about dynamics with multi-body interaction on hypergraphs.
In~\cite{shang2022system}, it is provided a first rigorous analysis of both discrete-time and continuous-time dynamics with gravity-like and Heaviside-like interactions on 3-uniform hypergraphs, i.e., with three-body interaction. However, the analysis shows that the investigated model boils down to the linear case where the consensus value is the average of the initial states of the \nodes.
In \cite{10.1093/comnet/cnad009}, a nonlinear consensus dynamics on temporal hypergraphs with random and noisy three-body interactions is analyzed, and sufficient conditions for all \nodes in the network to reach consensus are provided.

\paragraph{Our contribution}
In this paper, we analyze a class of discrete-time nonlinear dynamics with three-body interactions on hypergraphs that are inspired by the consensus model introduced in~\cite{PhysRevE.101.032310}. 
The class of dynamics we consider reads as follows
\begin{equation}\label{eq: dynamics formulation intro}
	\state{t+1}_i = \state{t}_i + \sum_{ \{j,k\} \in \nbr{i}} \frac{s(\lambda|\state{t}_j-\state{t}_k|)}{\snorm{t}_i} \cdot
	\frac{(\state{t}_j-\state{t}_i)+(\state{t}_k-\state{t}_i)}{2},
\end{equation}
where we denote by $\state{t}_i$ the state of the \node $i$ at time $t$ and $s(\lambda|\state{t}_j-\state{t}_k|)$ encodes the nonlinear interaction; moreover, $N_i$ is the set of neighboring pairs connected to \node $i$.
The dynamics is described in full detail in \cref{eq: dynamics formulation}.
Roughly speaking, the process goes as follows. 
A set of $n$ agents are interconnected by a 3-uniform hypergraph, i.e., hyperedges are triples of \nodes.
Analogously to the DeGroot model, in each discrete-time step every agent updates its opinion, or \textit{state}, by performing an average of the states of its neighboring pairs; however, the average is weighted by a nonlinear function of the neighbors states that is not present in the DeGroot model. 
We consider a large class of nonlinear functions that include, for example, some modeling choices made in~\cite{PhysRevE.101.032310} to capture a group reinforcement effect.

By considering a mean-field model where \nodes are interconnected by a trivial topology, we first analytically derive the consensus value the agents converge to (\cref{prop: trivial topology}).
Then, as our main result (\cref{theorem: main}), we prove that the mean-field model is a good approximation of the actual dynamics, whenever the underlying hypergraph satisfies some density and regularity assumptions and the initial states of the agents follow some probability distribution. 
In fact, we prove that the leading order of the consensus value is the same as that of the mean field model, with high probability.
Additionally, we bound the convergence time of the dynamics as a function of the hypergraph's structure.

As empirically observed in~\cite{PhysRevE.101.032310}, the dynamics we analyze converge to a consensus value that, when compared to the DeGroot model, is shifted.
Quite remarkably, we exactly characterize such a shift and rigorously prove its dependency on the initial states and on the three-body interaction function, which have an effect on both the magnitude and direction of the shift (\cref{eq: shift NML20}).

In order to prove our main result, we proceed as follows. 
First, we exactly characterize the nonlinear effect after the first interaction, that produces a multiplicative shift from the average of the initial states of the \nodes.
Second, we perform a Taylor expansion of the dynamics up to the second order to keep track of the dependency of the nonlinear function in the update. 
Third, we analyze the linear and nonlinear terms separately: we prove that the linear contribution is responsible for the convergence to a consensus, while the leading order of the shift is produced after the first interaction; we prove that the residual nonlinear contribution deriving from subsequent interactions, instead, is negligible.
We remark that our analysis holds with high probability: we exploit the randomness of the initial states of the \nodes both to characterize the shift in the consensus value and to prove that the residual nonlinear contribution is of a smaller order.

Moreover, we prove that there exist hypergraphs that satisfy the assumptions of our main theorem (\cref{prop: erh}).
In particular, such a class of hypergraphs is a natural extension of Erd\H{o}s-R\'{e}nyyi random graphs.
In order to prove that they satisfy the assumptions of our main theorem, of independent interest, we provide bounds on the relation of the spectra of different matrix representations of a graph.

Finally, we run simulations hinting that our result might hold even for hypergraph topologies that are not covered by \cref{theorem: main}.

\subsection{Link to diffusion equations}\label{sec: heat}

This subsection establishes a connection between the diffusion equation and our proposed model. 
In order to accomplish this, we consider a graph where the set of \nodes is the $d$-dimensional integer lattice or the quotient space of the $d$-dimensional integer lattice. 
The two spaces will have roles to derive partial differential equations on the $d$-dimensional real coordinate space $\mathbb{R}^d$ and on the $d$-dimensional torus $\mathbb{T}^d$, respectively. 
We approach this section similarly to how it is done in numerical methods (see, e.g., \cite{doi:10.1137/9780898719987}).
In particular, we will examine two distinct cases: linear and nonlinear interactions.
In the former, the dynamics reduces (see \cite[Chapter 20]{press2007numerical}) to a linear dynamics on a graph, studied in \cref{prop: linear interaction}.
In the latter, although there are no results yet, we expect that the nonlinear dynamics on hypergraphs is a discretized version of a nonlinear heat equation, which is explained in \cref{sec:nonLinearinteractions}. 
By numerical simulations (see \cref{sec: simulations}), we can observe that the convergence behavior of the nonlinear dynamics on hypergraphs is similar to that described in \cref{theorem: main}.
However, the assumptions of \cref{theorem: main} are not satisfied by the lattice hypergraph used in the discretization. 
Therefore we cannot translate our convergence result to the nonlinear diffusion equation as we can in the linear case.
For future perspectives, one can then try to relax the assumptions on the hypergraphs in our main result \cref{theorem: main}
in order to study the aforementioned nonlinear diffusion equation.

\subsubsection{Linear interactions}\label{sec:Linearinteractions}
Consider a lattice graph $G=(V,E)$ with $V=(h\mathbb{Z})^d$ for some $h>0$ (or, similarly, $V=(h\mathbb{Z}/Lh\mathbb{Z})^d$ for some positive integer $L$) and $E=\big\{\{i,i+e_j\} : i\in V,\, e_j \in \{0,1\}^d\big\}$, where $e_j$ denotes the $j$-th element of canonical basis of $\mathbb{R}^d$ (or $\mathbb{T}^d$).
Consider the dynamics given in \cref{eq: dynamics formulation intro} with $s(x)=1$ or, equivalently, with $\lambda=0$ (see \cref{assumption: function s}); 
by using the definition of the edge set $E$ on the lattice graph,  we get
\begin{equation}\label{eq: linear heat}
    \state{t+1}_i - \state{t}_i 
    = \frac{1}{2d} \sum_{ j \in \nbr{i}} \frac{\state{t}_j-\state{t}_i}{2}
    = \frac{1}{d} \sum_{j=1}^{d} \frac{\frac{\state{t}_{i-e_j}-\state{t}_i}{2} - \frac{\state{t}_i - \state{t}_{i+e_j}}{2}}{2},
\end{equation}
where we used that $j \in N_i$ implies that $x_j^{(t)} - x_i^{(t)} = x_{i-e_j}^{(t)} - x_i^{(t)}$.
By rescaling the time with respect to $(h/2)^2$ and the size of the lattice with respect to $(h/2)$, and by taking the limit $h\to 0$, we get the following diffusion equation, also known as \textit{heat equation} on $\mathbb{R}^d$ (or $\mathbb{T}^d$):
\begin{equation}\label{eq: heat-equation}
\partial_{t} x(t,z)=-C_d\,\Delta_z x(t,z)
\end{equation}
where $x:\mathbb{R}\times\mathbb{R}^d\,(\text{or }\mathbb{R}\times\mathbb{T}^d) \to \mathbb{R}$ is a solution of the diffusion equation, $C_d$ is a constant depending on the dimension, and $\Delta_z$ denotes the Laplace operator.

In essence, in the aforementioned limit, the dynamics in \cref{eq: heat-equation,eq: dynamics formulation intro} are equivalent when the underlying graph $G$ is the $d$-dimensional lattice or torus. 
Moreover, in the case of a compact domain (e.g., in $\mathbb{T}^d$), the solution of the heat/diffusion equation ultimately attains equilibrium, resulting in a uniform temperature distribution across all positions. 
Such a behavior corresponds to the consensus reached by the agents in the DeGroot model as described in our results with \cref{prop: linear interaction,theorem: main}.

\subsubsection{Nonlinear interactions}
\label{sec:nonLinearinteractions}

Consider a lattice hypergraph $\Gamma=(V,H)$ with $V=\mathbb{Z}^d$ (or $(\mathbb{Z}/L\mathbb{Z})^d$) and $H=\big\{\{i-e_j,\,i,\,i+e_j\} : i\in V,\, e_j\in \{0,1\}^d\big\}$. Here, as before, $e_j$ denotes the $j$-th element of the canonical basis of $\mathbb{R}^d$ (or $\mathbb{T}^d$). 
Similarly to the linear case, consider the dynamics given in \cref{eq: dynamics formulation}, but with a nonlinear function $s$ or, equivalently, with $\lambda \neq 0$ (see \cref{assumption: function s}); 
by rearranging the terms, we have
\begin{align*}
    \state{t+1}_i - \state{t}_i
    &= \frac{1}{d} \sum_{ \{j,k\} \in \nbr{i}} 
    s\big(\lambda|\state{t}_j-\state{t}_k|\big) \cdot
    \frac{(\state{t}_j-\state{t}_i) - (\state{t}_i - \state{t}_k)}{2}
    \\
    &= \frac{1}{d} \sum_{j=1}^{d} 
    s\big(\lambda|\state{t}_{i-e_j}-\state{t}_{i+e_j}|\big) \cdot 
    \frac{\frac{\state{t}_{i-e_j}-\state{t}_i}{2} - \frac{\state{t}_i - \state{t}_{i+e_j}}{2}}{2},
\end{align*}
where we proceeded similarly as in \cref{eq: linear heat}.
To simplify the following discussion, we assume that $s$ is a non-negative increasing function.
Similarly to the previous analysis (refer to the comparison between \cref{eq: linear heat,eq: heat-equation}), we can approximate continuous differentials by discrete differences. This can be done by the introduction of a scaling parameter $h$, which scales both time and distances between \nodes, as in \cref{sec:Linearinteractions}. 
Note that, in this context, we also need to rescale $\tilde{\lambda} = 2h \lambda$, to have comparable dynamics. Then, by taking formal limit as $h \to 0$, we obtain
\[
    \partial_{t} x(t,z)
    = -C_d\sum_{j=1}^d s(\tilde{\lambda} |\partial_{z_j} x(t,z)|)\,\partial_{z_j}^2 x(t,z),
\]
where $x:\mathbb{R}\times\mathbb{R}^d\to \mathbb{R}$ (or $\mathbb{R}\times\mathbb{T}^d \to \mathbb{R}$) is a solution of diffusion equation and $C_d$ is a constant depending on the dimension $d$ of the space.
Note that this is a parabolic equation with a modified Laplace operator. To be more specific, the Laplacian is modified as follows:
\[
    \Delta_z := \sum_{j=1}^{d} \partial_{z_j}^2 \quad\longrightarrow\quad
    \widetilde{\Delta}_z := \sum_{j=1}^{d} s(\tilde{\lambda} |\partial_{z_j} \cdot |) \partial_{z_j}^2
\]
and therefore the dynamics becomes
\begin{equation}\label{eq: tilde lapl}
    \partial_{t} x(t,z) = -C_d \, \widetilde{\Delta}_z x(t,z).
\end{equation}

In this case, $\widetilde{\Delta}_z$ modifies the diffusion rate that is now governed by a function of the absolute value of the gradient.
When $\lambda<0$, the diffusion process gives more weight to the regions with larger support and small moduli of the temperature gradients.
This potentially results in a degenerate solution; hence, it might be hard to prove the well-posedness of the partial differential equation. 
In contrast, if $\lambda>0$ the diffusion gives more weight to the directions where the moduli of the temperature gradients are big.
This leads to a standard quasi-linear parabolic partial differential equation.
Employing the maximum principle and the comparison principle, one could derive that, if $\|x(t,\cdot)\|_{L^\infty} < C$ uniformly in time $t$, there exists a unique solution depending on the given initial data. 
However, to the best of our knowledge, this has not been studied in the partial differential equations literature and could be an interesting open research direction.

Considering the link between \cref{eq: dynamics formulation intro} and \cref{eq: tilde lapl}, we expect the diffusion dynamics on $\mathbb{R}^d$ or $\mathbb{T}^d$ to have a similar behavior to that described in our main result.
In particular, we expect to observe a shift from the average temperature in the equilibrium distribution whose direction depends on the sign of $\lambda$ and that could generate heat gain/loss. 
However, a direct translation of the concepts of initial majority and balance (i.e., the directions of the shift described by our main result) to the continuous domain is not immediate.

There exists related literature in the area of partial differential equations.
For example, there are many works considering the following advection-diffusion equation (e.g., \cite{brezis1979uniqueness,Volpert_1969,Jennifer2021}),
\[
    \partial_t x(t,z) + \nabla_z \big(\, f(x(t,z),z) - k(x(t,z))\, \nabla_z x(t,z) \,\big) = 0.
\]
If one considers $k(x(t,z)) = s( \lambda |\nabla_z x(t,z)|)$ and $\nabla_z \, f(x(t,z),z) = \nabla_z\,k(x(t,z)) \cdot \nabla_z\, x(t,z)$, then the advection-diffusion equation turns similar to the limiting dynamics given above.

\subsection{Link to graph neural networks}\label{sec: gnns}

In this section, we provide an overview of some applications of nonlinear averaging and diffusion processes on hypergraphs, focusing on their significance for \textit{graph neural networks} (GNNs). 
GNNs are semi-supervised machine learning methods that process data represented as graphs and can be used for \node/edge classification tasks, see, e.g., \cite{hamilton2020graph}. 
They sequentially update a set of features of the graph through a message-passing framework that can be described as follows.
Let $G=(V,E)$ be a graph and, for each \node $i\in V$, let $N_{i}$ be the set of neighbors of \node $i$. Additionally, let $\mathbf {x}_{i}$ be the vector of features of \node $i \in V$. 
Graph- and edge-level features can be also considered, but we omit them for simplicity.
An update for a \node $i\in V$ in a GNN is expressed as follows:
\[
    \mathbf{h}_{i} = \phi \left(
        \mathbf{x}_{i}, 
        \bigoplus_{j\in N_{i}} \psi (\mathbf{x}_{i}, \mathbf{x}_{j} )
    \right)
\]
where $\phi$ is an \textit{update function}, $\psi$ is a \textit{message function}, and $\oplus$ is a permutation-invariant \textit{aggregation operator}.%
\footnote{Here, the symbol $\oplus$ does not simply denote a direct sum, but it also includes an aggregation function.}
The functions $\phi$ and $\psi$ are differentiable functions and can also be implemented by artificial neural networks. 
The operator $\oplus$ can accept an arbitrary number of inputs and can be, e.g., an element-wise sum, a mean, or a max.
Note that $\mathbf{h}_{i}$ is a new representation of the features of \node $i$. Hence the message-passing framework of the GNN can be iterated by applying the previous equation to the new feature vectors, resulting in a discrete-time dynamics over the \nodes of the graph.

Such a framework can naturally be extended to hypergraphs.
Formally, let $\Gamma=(V,H)$ be a hypergraph. 
For clarity of exposition, we describe it for 3-uniform hypergraphs, i.e., such that each hyperedge has cardinality three. For each \node $i\in V$, let $N_{i}$ be the set of neighboring pairs of \node $i$.
Calling $\mathbf {x}_{i}$ the vector of features of \node $i \in V$, the update of $i$ can be described as:
\begin{equation}\label{eq: hgnn}
    \mathbf{h}_{i} = \phi \left(
        \mathbf{x}_{i}, 
        \bigoplus_{\{j,k\}\in N_{i}} \psi (\mathbf{x}_{i}, \mathbf{x}_{j}, \mathbf{x}_{k} )
    \right)
\end{equation}
where $\phi,\psi,\oplus$ are as previously described.
Analogously to the graph setting, $\mathbf{h}_{i}$ represents a new vector of features of \node $i$ and, hence, the update can be iterated. The iteration describes a discrete-time dynamics over the hypergraph.

There exist several examples of scientific literature that consider nonlinear diffusion processes on hypergraphs, fitting the GNN framework, in several contexts.
For example, for opinion dynamics (as described in the introduction)~\cite{PhysRevE.101.032310} and chemical reaction networks~\cite{doi:10.1080/00207179.2015.1095353} using exponential and logarithmic message functions; 
for network oscillators~\cite{schaub2016graph} using trigonometric message functions;
for semi-supervised machine learning~\cite{10.1145/3442381.3450035,pmlr-v162-prokopchik22a} using polynomials message functions.

The dynamics that we analyze in this paper, described in \cref{eq: dynamics formulation intro}, also fit in the GNN message passing framework.
In particular, consider the following update, message, and aggregator functions (for some fixed function $s$ and parameter $\lambda$): 
\begin{enumerate}
    \item $\phi(\mathbf{x}, \mathbf{y}) = \mathbf{x}+\mathbf{y}$;
    \item $\displaystyle \psi(\mathbf{x}_i, \mathbf{x}_j, \mathbf{x}_k) = \frac{s(\lambda |\mathbf{x}_j- \mathbf{x}_k|)}{z_i} \cdot \frac{(\mathbf{x}_j- \mathbf{x}_i)+(\mathbf{x}_k-\mathbf{x}_i)}{2}$, 
    where $z_i = \sum_{\{j,k\}\in \nbr{i}} s(\lambda |\mathbf{x}_j- \mathbf{x}_k|)$;
    \item $\displaystyle \bigoplus_{\{j,k\}\in\nbr{i}} \psi(\mathbf{x}_i, \mathbf{x}_j, \mathbf{x}_k) = \sum_{\{j,k\}\in\nbr{i}} \psi(\mathbf{x}_i, \mathbf{x}_j, \mathbf{x}_k)$.
\end{enumerate}
Note that the above choice of functions $\phi,\psi,\oplus$ make \cref{eq: dynamics formulation intro} and \cref{eq: hgnn} coincide, i.e., our dynamics can be interpreted as a GNN.
Especially, we remark that our analysis works for functions $s$ that are analytic at 0 (see \cref{assumption: function s}). 
This implies that the standard softmax function can be considered as message function $\psi$ in our dynamics (see \cref{rem: exponential}), as well as many other classes of functions such as exponential, logarithmic, trigonometric, and polynomials.

\subsection{Structure of the paper}
The remainder of this article is structured as follows.
In \cref{sec:notation definitions} we describe the dynamics and introduce basic notations.
In \cref{sec: lin dyn} we briefly review known results for the linear counterpart of the dynamics.
\cref{sec: nonlinear dyn} is the heart of the article and where we consider the actual nonlinear dynamics. We start by providing a simple mean-field analysis on a trivial hypergraph topology. Then we consider the dynamics on actual hypergraphs, state our main theorem, and discuss the needed assumptions.

The remaining sections delve into the technical details.
In \cref{sec: nonlinear effect} we characterize the higher-order effect deriving from the interplay between the hypergraph topology and the nonlinearity in the update.
In \cref{sec: rademacher initial state} we discuss the effect of the randomness of the initial states.
In \cref{sec: erdos renyi hypergraph}, we show the existence of a class of hypergraphs that satisfy the assumptions of our main theorem.
Finally, in \cref{sec: simulations} we run numerical simulations of the dynamics on two hypergraph topologies.
\section{Notation and definitions}\label{sec:notation definitions}

In this section, we formally describe the structure and dynamics analyzed in the paper and we introduce further notation and definitions used throughout the paper.

\paragraph{Hypergraph}
Let $\Gamma=(V,H)$ be a hypergraph, where $V = \{1,\ldots,n\}$ is the set of \nodes and $H$ is the set of hyperedges, i.e., each hyperedge $e\in H$ is such that $e \subseteq V$ and $e \neq \emptyset$.
Herein we assume $\Gamma$ to be $3$-uniform, namely each hyperedge $e \in H$ has cardinality $|e|=3$; and we assume $\Gamma$ to be simple, namely each hyperedge $e\in H$ consists of distinct \nodes.
We call $A$ the adjacency tensor of $\Gamma$, namely the 3-dimensional matrix such that
\[
A_{ijk} = \begin{cases}  
1 &\text{if } \{i,j,k\}\in H, 
\\ 
0  &\text{otherwise.}
\end{cases}
\]
Note that, since $\Gamma$ is simple, for every $i,j$ it holds that $A_{iii} = A_{iij} = A_{iji} = A_{jii} = 0$ (since the hyperedges are made of distinct \nodes)
and that, for every $i,j,k\in V$, $A_{ijk}=A_{ikj}=A_{jik}=A_{jki}=A_{kij}=A_{kji}$ (since the hyperedges are unordered sets).
Moreover, for each \node $i \in V$ we call the neighborhood of $i$ the set
\begin{equation}\label{eq: def Ni}
    \nbr{i} := \{ \{j,k\} : \{i,j,k\} \in H \},
\end{equation}
namely the set of pairs of \nodes that share a hyperedge with $i$.
For each \node $i \in V$, we call $|\nbr{i}|$ the degree of \node $i$.

\paragraph{Dynamics}
We denote by $\state{t}_i$ the state of \node $i$ at time $t$.
Starting from an initial state $\state{0}_i$, the dynamics evolves in synchronous discrete-time steps in which all \nodes simultaneously update their state according to a function of the states of their neighbors.
Formally, for a fixed parameter%
\footnote{The range of $\lambda$ is restricted for the sake of simplicity. See \cref{remark: range of lambda} for a more detailed discussion on this.}
$\lambda \in (-1/2,1/2)$, in each time step $t$ each \node $i \in V$ decides its state in the next step as follows:
\begin{equation}\label{eq: dynamics formulation}
    \state{t+1}_i = \state{t}_i + \sum_{ \{j,k\} \in \nbr{i}} \frac{s(\lambda|\state{t}_j-\state{t}_k|)}{\snorm{t}_i} \cdot
    \frac{(\state{t}_j-\state{t}_i)+(\state{t}_k-\state{t}_i)}{2},
\end{equation}
where $s: \mathbb{R} \rightarrow \mathbb{R}$ is a function that models the strength of the interaction between \node $i$ and its pair of neighbors $\{j,k\}$, 
while $\snorm{t}_i$ is a normalization factor, i.e., 
\begin{equation}\label{eq:normalization}
    \snorm{t}_i \coloneqq \sum_{ \{j,k\} \in \nbr{i}} s(\lambda|\state{t}_j-\state{t}_k|).
\end{equation}
\begin{assumption}[Assumptions on the nonlinear function $s$]\label{assumption: function s}
We assume the function $s$ appearing in \cref{eq: dynamics formulation,eq:normalization} to be analytic at 0, namely with a convergent power series expansion such that 
\[
    s(x)=\sum_{k=0}^\infty a_k x^k,
\]
for coefficients $a_0,a_1,\dots$
In this setting, it is natural to consider $a_0=1$ because when $\lambda=0$ we want $s\equiv 1$ so that the dynamics has no three-body interaction forces.
Moreover, for the sake of simplicity, we also assume $a_1=1$ given that a different value only results in a rescaling. We don't make assumptions on the other coefficients since we will use the expansion up to the second order.
\end{assumption}

\begin{remark}\label{rem: exponential}
There exist three-body interaction strengths modeled by functions that are analytic at 0.
For example, in~\cite{PhysRevE.101.032310}, it is chosen $s(x)=e^{x}$ or, to be more precise, for $i \in V$ and $\{j,k\} \in \nbr{i}$ we have
\[
    s(\lambda|\state{t}_j-\state{t}_k|) = e^{\lambda|\state{t}_j-\state{t}_k|}.
\]
\end{remark}

Note that the dynamics described in \cref{eq: dynamics formulation} can be equivalently written as follows:
\begin{equation}\label{eq:dynamics clean}
    \state{t+1}_i
    = \frac{1}{\snorm{t}_i} \sum_{ \{j,k\} \in \nbr{i}} {s(\lambda|\state{t}_j-\state{t}_k|)} \cdot \frac{\state{t}_j+\state{t}_k}{2}.
\end{equation}
When the dynamics is written in this form, it appears more evident that the \nodes update their states by performing a weighted average, where the weights are a nonlinear function that depends on the states of their neighbors and which are evolving in every time step.

In this paper, we consider a binary initialization, namely $\state{0}_i \in \{-1,+1\}$, for every $i\in V$.
Note that the choice of the binary values to be in $\{-1,+1\}$ is not a restriction of the analysis, in the following sense.

\begin{proposition}
Let $\mathsf{s}_\downarrow , \mathsf{s}_\uparrow \in \mathbb{R}$ and such that $\mathsf{s}_\downarrow < \mathsf{s}_\uparrow$.
For every $i\in V$, let $y^{(0)}_i\in \{\mathsf{s}_\downarrow ,\mathsf{s}_\uparrow\}$ and let $\tilde\lambda \in (-\frac{1}{\mathsf{s}_\uparrow-\mathsf{s}_\downarrow },\frac{1}{\mathsf{s}_\uparrow-\mathsf{s}_\downarrow })$.
Let 
\[ 
    \state{0}_i = \frac{2}{|\mathsf{s}_\downarrow -\mathsf{s}_\uparrow|} \cdot \Big(y^{(0)}_i - \frac{\mathsf{s}_\downarrow +\mathsf{s}_\uparrow}{2}\Big),
    \qquad\text{and }
    \lambda = \tilde\lambda \cdot \frac{|\mathsf{s}_\downarrow - \mathsf{s}_\uparrow|}{2}.
\]
For every $i\in V$, let $y^{(t)}_i$ and $x^{(t)}_i$ update as in \cref{eq:dynamics clean} with parameters $\tilde\lambda$ and $\lambda$, respectively. 
It holds that, for all $t\geq 0$
\begin{equation}\label{eq: y(t)}
    y^{(t)}_{i}  = \frac{|\mathsf{s}_\downarrow - \mathsf{s}_\uparrow|}{2}\cdot \state{t}_i + \frac{\mathsf{s}_\downarrow +\mathsf{s}_\uparrow}{2}.
\end{equation}
\end{proposition}

\begin{proof}
Note that $\state{0}_i$ and $\lambda$ are such that $\state{0}_i \in \{-1,+1\}$ and $\lambda \in (-1/2,1/2)$.
Following the definition of $\snorm{t}_i$ in \cref{eq:normalization}, let us define
\[
    \tilde{z}^{(t)}_i := \sum_{\{j,k\}\in \nbr{i}} {s(\tilde\lambda|y^{(t)}_j-y^{(t)}_k|)}
\]
and note that
\[
    \tilde{z}^{(0)}_i = \sum_{\{j,k\}\in \nbr{i}} {s(\tilde\lambda|y^{(0)}_j-y^{(0)}_k|)}
    =\sum_{\{j,k\}\in \nbr{i}} {s(\lambda|\state{0}_j-\state{0}_k|)}
    = \snorm{0}_i.
\]
Hence, for all $t\geq 0$, we have
\begin{align*}
    y^{(1)}_{i} 
    &= \frac{1}{\tilde{z}^{(0)}_i} \left(
        \sum_{\{j,k\}\in\nbr{i}} {s(\tilde\lambda|y^{(0)}_j-y^{(0)}_k|)} \cdot  \frac{(y^{(0)}_{j}+y^{(0)}_{k})}{2}
    \right)
    \\
    &= \frac{1}{\snorm{0}_i} \left(
        \sum_{\{j,k\}\in\nbr{i}} {s(\lambda|\state{0}_j-\state{0}_k|)} \cdot\Bigg( \frac{|\mathsf{s}_\downarrow - \mathsf{s}_\uparrow|}{2} \cdot  \frac{(\state{0}_{j}+\state{0}_{k})}{2}+\frac{\mathsf{s}_\downarrow+\mathsf{s}_\uparrow}{2}\Bigg)
    \right)
    \\
    &= \frac{|\mathsf{s}_\downarrow - \mathsf{s}_\uparrow|}{2}\cdot \frac{1}{\snorm{0}_i}\left(
        \sum_{\{j,k\}\in\nbr{i}} {s(\lambda|\state{0}_j-\state{0}_k|)} \cdot  \frac{(\state{0}_{j}+\state{0}_{k})}{2}
    \right)
    +\frac{\mathsf{s}_\downarrow+\mathsf{s}_\uparrow}{2}
    \\
    &= \frac{|\mathsf{s}_\downarrow - \mathsf{s}_\uparrow|}{2}\cdot \state{1}_i + \frac{\mathsf{s}_\downarrow+\mathsf{s}_\uparrow}{2}.
\end{align*}
Then, since the scaling relation between $y^{(1)}_i$ and $\state{1}_i$ remains the same as that at time $t=0$, one can also see that $\tilde{z}^{(1)}_i = \snorm{1}_i$.
By iteratively repeating the calculation above, one sees that the same holds for every $t\geq 0$. Hence, 
\(
    y^{(t)}_{i}  = \frac{|\mathsf{s}_\downarrow - \mathsf{s}_\uparrow|}{2}\cdot \state{t}_i + \frac{\mathsf{s}_\downarrow+\mathsf{s}_\uparrow}{2}.
\)
\end{proof}

\paragraph{Motif graph}
Given a hypergraph $\Gamma=(V,H)$ with adjacency tensor $A$, we associate to it an edge-weighted graph $G=G(\Gamma)$ on the same \node set $V$ whose weighted adjacency matrix $W$ is defined as:
\begin{equation}\label{eq:graph W}
    W_{ij} := \sum_{k \in V} A_{ijk},
\end{equation}
namely each edge $\{i,j\}$ in $G$ is weighted by the number of \nodes that share an hyperedge with $i$ and $j$. 
Since $\Gamma$ is 3-uniform, the weight on $W_{ij}$ counts the number of triangles the \nodes $i,j$ belong to. Triangles are a typical pattern, a.k.a.\ \textit{motif}, appearing in graphs and for this reason, the graph $G$ is called the (triangle) \textit{motif graph} of $\Gamma$.
Note that $W$ is symmetric since $G$ is undirected, i.e., the hyperedges are unordered sets.

Moreover, we also associate to $G$ its diagonal degree matrix $D$, whose entry $D_{i\ell}$ is defined as
\begin{equation}\label{eq:graph D}
    D_{i\ell} := \delta_{i \ell} \cdot \sum_{j \in V} W_{ij}
    = \delta_{i \ell} \cdot \sum_{j,k \in V} A_{ijk}
    = \delta_{i \ell} \cdot 2|\nbr{i}|,
\end{equation}
where $\delta_{i \ell}$ is the Kronecker delta function and where the factor $2$ is due to the fact that the double sum counts ordered pairs, while the neighborhood consists of unordered pairs.
Note that if $\Gamma$ is $d$-regular (i.e., the neighborhood of every \node has size $d$) then $D = 2dI$, where $I$ is the identity matrix. 

Given the definitions of $W$ and $D$, or simply considering the (weighted) adjacency and degree matrices of any graph $G$, we can define the transition matrix of a random walk on $G$ as
\begin{equation}\label{eq: matrix P}
    P := D^{-1}W.
\end{equation}
Note that since $P$ is a row-stochastic matrix, i.e., every row sums up to $1$, its eigenvalues are $1 = \lambda_1(P) \ge \ldots \ge \lambda_n(P) \ge -1$.
Let us denote by $\nu$ the second largest (in absolute value) eigenvalue of $P$, i.e., 
\begin{equation}\label{eq: def nu}
    \nu := \max(|\lambda_2(P)|,\,|\lambda_n(P)|).
\end{equation}
Recall that $G$ is connected if and only if $\lambda_2(P) < 1$ and $G$ is not bipartite if and only if $\lambda_n(P) > -1$ (see, e.g., \cite{brouwer2011spectra}).
Hence, if $G$ is connected and not bipartite then $\nu<1$.

Let 
\[
    d_{\max} := \max_{i \in V} D_{ii} 
    \qquad\text{and }\quad
    d_{\min} := \min_{i \in V} D_{ii}.
\]
We define the ratio between maximum and minimum degrees as
\begin{equation}\label{eq: def delta}
    \Delta := \frac{d_{\max}}{d_{\min}}.
\end{equation}
Note that $\Delta$ is well defined for every connected graph since $\min_{i \in V} D_{ii}>0$.

\paragraph{Asymptotic notation}
Herein, we will make use of the Bachmann--Landau asymptotic notation, in the limit of $n\to\infty$.
Let $f,g : \mathbb{R}^+ \rightarrow \mathbb{R}^+$ be two positive real functions. We say that:
\begin{itemize}
    \item $f(n) = O(g(n))$ if $\lim\sup_{n\to\infty} \frac{f(n)}{g(n)} < \infty$ ($\exists k>0\, \exists N>0\, \forall n>N : \frac{f(n)}{g(n)} \le k$);
    \item $f(n) = \Omega(g(n))$ if $\lim\inf_{n\to\infty} \frac{f(n)}{g(n)} > 0$ ($\exists k>0\, \exists N>0\, \forall n>N : \frac{f(n)}{g(n)} \ge k$);
    \item$f(n) = \Theta(g(n))$ if $f(n)=O(g(n))$ and $f(n)=\Omega(g(n))$\\($\exists k_1,k_2>0\, \exists N>0\, \forall n>N : k_1 \le \frac{f(n)}{g(n)} \le k_2$);
    \item$f(n) = o(g(n))$ if $\lim_{n\to\infty} \frac{f(n)}{g(n)} = 0$ ($\forall k>0\, \exists N>0\, \forall n>N : \frac{f(n)}{g(n)} < k$);
    \item $f(n) = \omega(g(n))$ if $\lim_{n\to\infty} \frac{f(n)}{g(n)} = \infty$ ($\forall k>0\, \exists N>0\, \forall n>N : \frac{f(n)}{g(n)} > k$).
\end{itemize}%
Equivalently, we will use $f \ll g$ to denote that $f =o(g)$ and $f \gg g$ to denote that $f = \omega(g)$.
Moreover, whenever we write $x = f\pm g$ we mean that $|x-f| \le g$.
Note that the use of limits in the Bachmann--Landau asymptotic notation means that there exists a large number $N>0$ such that each notation holds for all $n>N$. 
In other words, our use of the notation holds for finite hypergraphs with a sufficiently large number of \nodes $n$.

\paragraph{Vector and operator norms}
We denote by $\mathbf{1}\in \mathbb{R}^n$ the vector of all ones.
For any vector $\mathbf{y}\in \mathbb{R}^n$, we denote its $\ell^\infty$-norm as $\|\mathbf{y}\|_\infty := \sup_{i\in \{1,\ldots, n\}} |y_i|$ and its $\ell^2$-norm as $\|\mathbf{y}\|_2:= (\sum_{j=1}^n |y_j|^2)^{\frac{1}{2}}$.
For any operator $\mathcal{A}: \mathbb{R}^n \mapsto \mathbb{R}^n$, we denote its $\ell^\infty$-norm as
$\|\mathcal{A}\|_\infty := \sup_{\mathbf{y} : \|\mathbf{y}\|_\infty=1} \|\mathcal{A}(\mathbf{y})\|_\infty$ and its $\ell^2$-norm as
$\|\mathcal{A}\|_2 := \sup_{\mathbf{y} : \|\mathbf{y}\|_2=1} \|\mathcal{A}(\mathbf{y})\|_2$.

\paragraph{Probability}
We say that an event $\mathcal{F}_n$, for $n\to\infty$, holds \textit{with high probability} if $\Pr(\mathcal{F}_n) = 1 - o(1)$.

\begin{definition}[$p$-Rademacher random vector]\label{def: p rademacher vector}
Let $p\in [0,1]$. We say that $\mathbf{x} \in \{-1,+1\}^n$ is a \emph{$p$\nobreakdash-Rademacher random vector} if each entry $x_i$ is such that $\Pr(x_i=+1)=p$ and $\Pr(x_i=-1)=1-p$, independently for every $i\in\{1,\ldots,n\}$. When we refer to $\mathbf{x}$ simply as a \emph{Rademacher random vector}, namely we omit the $p$, we mean that each entry $x_i$ is a Rademacher random variable, i.e., $p=1/2$.
\end{definition}

\begin{remark}[Norm of $\bstate{t}$]\label{rem: infty norm x t}
If the initial state of the dynamics $\bstate{0}$ is such that $\|\bstate{0}\|_\infty = 1$, e.g., if $\bstate{0}$ is a $p$-Rademacher random vector,
then $\bstate{t}$ satisfies $\|\bstate{t}\|_\infty \leq 1$ for every $t$, as a consequence of \cref{eq:dynamics clean}.
\end{remark}

\section{Linear interaction}\label{sec: lin dyn}

In this section, we provide an overview of some known facts about the linear counterpart of the nonlinear dynamics we consider in this paper.
When the interaction function $s$ in \cref{eq:dynamics clean} is linear and constant over time, the dynamics greatly simplifies and reduces to a linear weighted average (as proven in \cref{prop: linear interaction} and in \cite{PhysRevE.101.032310}). 
{This happens when $\lambda=0$ because of our assumptions on $s$ made in \cref{assumption: function s}, that imply $s(\lambda x)=s(0)=1$.}

\begin{proposition}[Linear interaction]\label{prop: linear interaction}
Let $\state{0}_i$ evolve according to \cref{eq:dynamics clean} with $\lambda= 0$.
It holds that
\[
    \state{t+1}_i = \frac{1}{D_{ii}}\sum_{j\in V} W_{ij}\,\state{t}_j,
\]
where $W=(W_{ij})$ is the adjacency matrix of the graph defined in \cref{eq:graph W} and $D=(D_{ii})$ is the diagonal degree matrix of the graph defined in \cref{eq:graph D}.
Equivalently, let $\bstate{0}$ the vector of \nodes states at time $t=0$. 
We have that
\[
    \bstate{t+1} = P\bstate{t} = P^{t+1}\bstate{0},
\]
where $P$ is the transition matrix defined in \cref{eq: matrix P}.
\end{proposition}
\begin{proof}
As noted above, because $\lambda=0$, we have $s(\lambda|\state{t}_j-\state{t}_k|) = 1$ for all \nodes $j,k$ and at every time $t$.
Starting from \cref{eq:dynamics clean}, by using \cref{eq:graph D} and the previous observation we get:
\begin{align*}
    \state{t+1}_i &= \frac{1}{|\nbr{i}|} \sum_{ \{j,k\} \in \nbr{i}} \frac{(\state{t}_j+\state{t}_k)}{2}
    = \frac{1}{D_{ii}} \sum_{j \in V}\sum_{k \in V} A_{ijk} \cdot \frac{(\state{t}_j + \state{t}_k)}{2}
    \\
    &= \frac{1}{D_{ii}} \Bigg[ 
          \frac{1}{2} \sum_{ j \in V} \Big(\sum_{ k \in V} A_{ijk}\Big) \state{t}_j 
        + \frac{1}{2} \sum_{ k \in V} \Big(\sum_{ j \in V} A_{ikj}\Big) \state{t}_k 
    \Bigg]
    = \frac{1}{D_{ii}}\sum_{j\in V} W_{ij}\,\state{t}_j.
\end{align*}
Note that such formulation is equivalent, in matrix notation, to $\bstate{t+1}=P\bstate{t}$. By iterating for every $t$ we get that $\bstate{t+1}=P\bstate{t}=P(P\bstate{t-1})=\ldots=P^{t+1}\bstate{0}$.
\end{proof}

Note that the average described in the previous proposition is unweighted if and only if $G$ is regular, i.e., when $\Gamma$ is regular.

Moreover, in the linear case, the behavior of the dynamics is well-characterized: it converges to the (weighted) average of the initial states of the \nodes at a rate that is governed by the second largest eigenvalue $\nu$ of the transition matrix of $G$, as described in the following proposition. This behavior will be exploited in our analysis.

\begin{proposition}[Linear Dynamics {\cite[Section~2.1 and Lemma~4.5]{BECCHETTI202049}}]
\label{prop: linear dynamics convergence}
Let $\bstate{0} \in \mathbb{R}^n$. Consider the linear averaging dynamics
\(
    \bstate{t} = P^t\bstate{0},
\)
where $P$ is the transition matrix of a random walk on a (possibly edge-weighted) graph $G$.
Let $\nu$ and $\Delta$ as in \cref{eq: def nu} and \cref{eq: def delta}, respectively.
For every $t>0$, it holds that
\begin{equation}\label{eq: def mu}
    \bstate{t} = \bar\mu \mathbf{1} + \mathbf{r}^{(t)},
    \quad\text{with }
    \bar\mu := \frac{\sum_{i\in V}D_{ii}x_i^{(0)}}{\sum_{i\in V} D_{ii}},
    \quad\text{and }
    \| \mathbf{r}^{(t)} \|_\infty \leq \nu^t \sqrt{\Delta n},
\end{equation}
where $\bar\mu$ is the weighted average of the initial states and $\nu, \Delta$ are defined in \cref{eq: def nu,eq: def delta}.
\end{proposition}

\cref{prop: linear dynamics convergence} says that the linear dynamics on a graph $G$, starting from an initial configuration of the states $\bstate{0}$, converges to the weighted average of the initial states, in a time which is governed by $\nu$ and therefore by the topology of the graph, e.g., by considering Cheeger's inequality (see, e.g., \cite{brouwer2011spectra}).

\section{Nonlinear three-body interaction}\label{sec: nonlinear dyn}

In this section, we consider the nonlinear dynamics.
In particular, in \cref{sec:mean field} we consider the dynamics on a mean-field model that allows us to simply characterize the convergence of the dynamics.
In \cref{sec: general case}, instead, we analyze the actual dynamics, by proving it converges to the same value of the mean field model and by providing bounds on the convergence time.

\subsection{Mean-field model}\label{sec:mean field}
In this (sub)section, we investigate the dynamics on a \emph{trivial topology}, i.e., namely on a hypergraph such that $A_{ijk}=1$ for every $(i,j,k)$ including repetitions (e.g., $i=j=k$).
\begin{proposition}[Trivial topology]\label{prop: trivial topology}
Let $\Gamma=(V,H)$ be a hypergraph with trivial topology and $|V|=n$.
Assume that $\bstate{0} \in \{-1,+1\}^n$.
Let $a = |\{i\in V:\state{0}_{i}=+1\}|$ and let $b = n-a$.
Then, for all $t\geq 1$
\begin{equation}\label{eq: dyn trivial top}
    \bstate{t} = \frac{a-b}{n} \left(
        1 +\frac{2ab\,(1-{s(2\lambda)})}{n^{2}-2ab\,(1-{s(2\lambda)})}
    \right) \cdot \mathbf{1}.
\end{equation}
\end{proposition}
\begin{remark}
Note that the coefficient $(a-b)/n$ appearing in \cref{eq: dyn trivial top} corresponds to the weighted average of the initial states $\bar\mu$, defined in \cref{eq: def mu}. Indeed under the assumption of \cref{prop: trivial topology}, the average simplifies to $\bar\mu = \frac{1}{n} \sum_{i \in V} \state{0}_i = \frac{a-b}{n}$.
\end{remark}
\begin{remark}
Note that the nonlinearity of the update rule made the system converge to a value that is multiplicatively shifted with respect to the average of the initial states $\bar\mu = (a-b)/n$. 
Moreover, we note that \cref{eq: dyn trivial top} is in agreement with the linear case when $\lambda = 0$.
\end{remark}
\begin{proof}[Proof of \cref{prop: trivial topology}]
We show that the dynamics converges in one single step since all \nodes perform the exact same update, namely $\state{1}_{i} = \state{1}_{j}$ for all $i,j \in V$. 
Note that, given the initialization $\state{0}_{i} \in \{-1,+1\}$ for all \nodes $i\in V$, we have that: 
\begin{enumerate}
    \item[(i)] if $\state{0}_{j}=\state{0}_{k}$, then $(\state{0}_{j}+\state{0}_{k})/2 = \state{0}_{j}$ and ${s(\lambda|\state{0}_j-\state{0}_k|) = 1}$;
    \item[(ii)] if $\state{0}_{j}\neq\state{0}_{k}$, then $(\state{0}_{j}+\state{0}_{k})/2 = 0$ and ${s(\lambda|\state{0}_j-\state{0}_k|) = s(2\lambda)}$.
\end{enumerate}
Therefore
\begin{align*}
    \state{1}_{i} 
    &= \frac{1}{\snorm{0}_i} \Bigg(
        \sum_{\substack{j,k\in V\\\state{0}_j=\state{0}_k}} s(\lambda|\state{0}_j-\state{0}_k|) \, \frac{\state{0}_{j}+\state{0}_{k}}{2}
        +\sum_{\substack{j,k\in V\\\state{0}_j\neq\state{0}_k}}  s(\lambda|\state{0}_j-\state{0}_k|) \, \frac{\state{0}_{j}+\state{0}_{k}}{2}
    \Bigg)
    \\
    &=\frac{1}{\snorm{0}_i} \sum_{\substack{j,k\in V\\\state{0}_j=\state{0}_k}} \state{0}_j .
\end{align*}
Using the equivalence $a+b=n$, we have that
\begin{align*}
    \state{1}_{i} 
    &=\frac{a^{2}-b^{2}}{a^{2}+b^{2}+2ab\cdot{s(2\lambda)}}
    =\frac{(a+b)(a-b)}{(a+b)^2-2ab(1-{s(2\lambda)})}
    =\frac{a-b}{n-\frac{2ab}{n}(1-{s(2\lambda)})}
    \\
    &=\frac{a-b}{n} \left(
        1 +\frac{2ab(1-{s(2\lambda)})}{n^{2}-2ab(1-{s(2\lambda)})}
    \right).
\end{align*}
\end{proof}

\subsection{General case}\label{sec: general case}
We now take into account the general case. 
In particular, we first rewrite the nonlinear dynamics in an equivalent formulation which is helpful in comparing the nonlinear evolution with that of the corresponding linear dynamics (see \cref{sec: dyn operator notation}). 
Later, we specify the assumptions we need on the hypergraph to prove \cref{theorem: main} and we state our main result (see \cref{sec: main result}). 
Before starting the discussion, we underline that, for such a comparison, it would be enough to approximate the nonlinear interaction via Taylor expansion to the first order. However, in order to keep track of the dependency on $s$ and $\lambda$ in our final approximation of the nonlinear interaction, we do Taylor expansion up to the second order (see \cref{pro: 2nd order taylor}).

\begin{proposition}\label{pro: 2nd order taylor}
Let $\lambda\in (-1/2,1/2)$.
For any $i\in V$, let $\state{0}_i$ evolve according to \cref{eq:dynamics clean}, i.e., 
\[
   x_i^{(t+1)} =  \frac{1}{\snorm{t}_i} \sum_{ \{j,k\} \in \nbr{i}} {s(\lambda|\state{t}_j-\state{t}_k|)} \cdot \frac{\state{t}_j+\state{t}_k}{2}.
\]
For every time $t$, it holds that
\begin{equation*}
    \state{t+1}_i = 
    \frac{1}{|\nbr{i}|} \sum_{\{j,k\} \in \nbr{i}} \!\! \left(
        1 
        + \frac{\lambda}{|\nbr{i}|} \sum_{\{\ell,m\} \in \nbr{i}} \!\!\left(
            |\state{t}_j-\state{t}_k| - |\state{t}_\ell-\state{t}_m| 
        \right)
        + \mathcal{E}_{\lambda;i,j,k}^{(t)}
    \right) \frac{\state{t}_j+\state{t}_k}{2},
\end{equation*}
with 
\begin{equation*}
    \mathcal{E}_{\lambda;i,j,k}^{(t)} = O\left(\lambda^2\max_{\{j,k\} \in \nbr{i}}|\state{t}_j-\state{t}_k|^2\right).
\end{equation*}
\end{proposition}
\begin{remark}[Restriction on the range of $\lambda$]\label{remark: range of lambda}
    For the sake of simplicity, we restrict $\lambda \in (-1/2, 1/2)$. 
    Such a restriction is enough to guarantee that $\lambda|x_j^{(t)} - x_k^{(t)}| < 1$ since $|x_j^{(t)} - x_k^{(t)}|\leq 2$ (see \cref{rem: infty norm x t}), which is important to get a meaningful Taylor expansion when approximating the nonlinear dynamics. However, our analysis also holds for larger values of $|\lambda|$, with high probability, since we only use Taylor expansion for $t>1$, when $\|\bstate{t}\|_\infty$ is already small.
\end{remark}
\begin{proof}[Proof of \cref{pro: 2nd order taylor}]
We write an expansion for the function $s$ appearing in \cref{eq:dynamics clean}.
{Using \cref{assumption: function s}}, for any $i\in V$, and for all the pairs of \nodes $\{j,k\}\in N_i$, we write
\begin{equation}\label{eq: approx exp}
    {s({\lambda|\state{t}_j - \state{t}_k|})} = 1 + \lambda |\state{t}_j - \state{t}_k| + O(\lambda^2|\state{t}_j - \state{t}_k|^2). 
\end{equation}
From \cref{eq:normalization,eq: approx exp} it follows that 
\begin{align}\label{eq: approx 1/z}
    \frac{1}{\snorm{t}_i} = \frac{1}{|N_i|} \Bigg(1 - \frac{1}{|N_i|}\sum_{\{j,k\}\in N_i} \left( 
        \lambda |\state{t}_j - \state{t}_k| + O(\lambda^2|\state{t}_j - \state{t}_k|^2)
    \right)\Bigg).
\end{align}
We then put the two expansions \cref{eq: approx exp} and \cref{eq: approx 1/z} in \cref{eq:dynamics clean} and we get 
\begin{equation*}
    \state{t+1}_i = 
    \frac{1}{|\nbr{i}|} \sum_{\{j,k\} \in \nbr{i}} \!\! \left(
        1 
        + \frac{\lambda}{|\nbr{i}|} \sum_{\{\ell,m\} \in \nbr{i}} \!\!\left(
            |\state{t}_j-\state{t}_k| - |\state{t}_\ell-\state{t}_m| 
        \right)
        + \mathcal{E}_{\lambda;i,j,k}^{(t)}
    \right) \frac{\state{t}_j+\state{t}_k}{2}
\end{equation*}

where $\mathcal{E}_{\lambda;i,j,k}^{(t)} := \sum_{\sharp=2}^{4} \mathcal{E}_{\lambda;i,j,k;\sharp}^{(t)}$ is the sum of the following smaller order error terms (where the index $\sharp$ denotes the order of the error).
In particular:
\begin{itemize}[leftmargin=*]
\item \(\displaystyle
    \mathcal{E}_{\lambda;i,j,k;2}^{(t)} := \frac{\lambda^2}{|N_i|}\!\!\sum_{\{\ell,m\}\in N_i} \!\!\left(
        -( |\state{t}_j \!-\! \state{t}_k| \cdot |\state{t}_\ell \!-\! \state{t}_m| )
        \!+\! O(|\state{t}_j \!-\! \state{t}_k|^2)
        \!+\! O(|\state{t}_\ell \!-\! \state{t}_m|^2)
    \right);
\)
\item \(\displaystyle
    \mathcal{E}_{\lambda;i,j,k;3}^{(t)} := \frac{\lambda^3}{|N_i|} \!\!\sum_{\{\ell,m\}\in N_i} \!\!\left(
        |\state{t}_j - \state{t}_k| 
        \cdot O(|\state{t}_\ell -\state{t}_m|^2)
        + |\state{t}_\ell - \state{t}_m|
        \cdot O(|\state{t}_j - \state{t}_k|^2)
    \right);
\)
\item \(\displaystyle
    \mathcal{E}_{\lambda;i,j,k;4}^{(t)} := \frac{\lambda^4}{|N_i|}\!\!\sum_{\{\ell,m\}\in N_i} \!\!\left( 
        O(|\state{t}_j - \state{t}_k|^2)
        \cdot O(|\state{t}_\ell -\state{t}_m|^2) 
    \right).
\)
\end{itemize}
Moreover, we have
\[
    \mathcal{E}_{\lambda;i,j,k}^{(t)} = O\left(\lambda^2\max_{\{j,k\} \in \nbr{i}}|\state{t}_j-\state{t}_k|^2\right),
\]
as a direct consequence of the definition of $N_i$ (see \cref{eq: def Ni}) and of the fact that $\|\bstate{t}\|_\infty \leq 1$ for all times $t$, as discussed in \cref{rem: infty norm x t}.
\end{proof}

\subsubsection{Dynamics in operator notation}\label{sec: dyn operator notation}
In our main result, we want to approximate the nonlinear dynamics \cref{eq:dynamics clean} with the corresponding linear dynamics discussed in \cref{sec: lin dyn}, while keeping track of the leading nonlinear contribution and on its dependence on $\lambda$. 
In particular, our goal is to use \cref{pro: 2nd order taylor} to write the dynamics as a perturbation of the linear dynamics, as already taken into account in \cref{sec: lin dyn}. 
For this purpose, for any fixed $\lambda\in (-1/2,1/2)$, we define two operators $\mathcal{P}: \mathbb{R}^n\rightarrow \mathbb{R}^n$ and $\mathcal{Q}_\lambda: \mathbb{R}^n\rightarrow \mathbb{R}^n$ such that $\mathcal{P}+\lambda \mathcal{Q}_\lambda$ is the operator for our update:
\begin{equation}\label{eq:nonlinear dynamics matrix apx}
    \bstate{t+1} = (\mathcal{P} + \lambda \mathcal{Q}_\lambda)(\bstate{t}).
\end{equation}
We define the two operators starting from \cref{pro: 2nd order taylor}:
$\mathcal{P}$ models the linear effect of the first order approximation in the expansion, while $\mathcal{Q}_\lambda$ models the nonlinear effect of the second and smaller orders in the expansion, and as a consequence depends on $\lambda$. 
In particular, we define $\mathcal{P}$ as a linear operator that acts on $\bstate{t}$ as the matrix $P$ (recall \cref{eq:graph W,eq:graph D,eq: matrix P}), i.e.,
\begin{equation}\label{eq: op P acts as a matrix}
    \mathcal{P}(\bstate{t})=D^{-1}W\bstate{t}=P\bstate{t}.
\end{equation}
We define $\mathcal{Q}_\lambda$ as a nonlinear operator that acts on $\bstate{t}$ so that the $i$-th entry of vector $\mathcal{Q}_\lambda(\bstate{t})$ is
\begin{equation}\label{eq:matrix Qt}
[\mathcal{Q}_\lambda(\bstate{t})]_{i} := 
\frac{1}{\lambda} \cdot \frac{1}{D_{ii}} \sum_{j\in V} \!\left[
    \sum_{k\in V} A_{ijk}\!\!\left(
        \frac{\lambda}{|\nbr{i}|} \!\sum_{\{\ell,m\}\in\nbr{i}}\!\!\!\left(
            |\state{t}_j\!-\!\state{t}_k| \!-\! |\state{t}_\ell\!-\!\state{t}_m|
    \right) + \mathcal{E}_{\lambda;i,j,k}^{(t)} \right)
\!\right] \cdot \state{t}_j
\end{equation}
where $1/\lambda$ is factored out for convenience, so that $\lambda$ appears in \cref{eq:nonlinear dynamics matrix apx},
and where we wrote the sum over pairs of neighbors of a \node $i$ as a sum over single \nodes (as seen in \cref{prop: linear interaction}).

By iterating \cref{eq:nonlinear dynamics matrix apx}
and using $\circ$ to denote the composition of operators, we get
\begin{equation}\label{eq:xt vector form}
    \bstate{t+1} =\underset{\text{$t$ times}}{\underbrace{(\mathcal{P}+\lambda \mathcal{Q}_\lambda)\,\circ\,\cdots\,\circ\,(\mathcal{P}+\lambda \mathcal{Q}_\lambda)}}\,(\bstate{1})
    = (\mathcal{P}+\lambda \mathcal{Q}_\lambda)^{t} (\bstate{1}).
\end{equation}
Note that we do not iterate until $\bstate{0}$, but instead, we stop at $\bstate{1}$. 
Later, by \cref{lemma: value of xi(1)}, we fully characterize $\bstate{1}$ keeping track of the leading nonlinear contribution to the dynamics.

In conclusion, defining
\begin{equation}\label{eq:matrix Rt}
    \mathcal{R}^{(t)}_\lambda := (\mathcal{P}+\lambda \mathcal{Q}_\lambda)^t-\mathcal{P}^t,
\end{equation}
by plugging \cref{eq:matrix Rt} into \cref{eq:xt vector form}, and by using \cref{eq: op P acts as a matrix}, we rewrite the dynamics as follows:
\begin{equation}\label{eq: dynamics final form}
    \bstate{t+1} = \mathcal{P}^t(\bstate{1}) + \mathcal{R}^{(t)}_\lambda (\bstate{1}) = P^t\bstate{1} + \mathcal{R}^{(t)}_\lambda (\bstate{1}).
\end{equation}

Note that, after the first update, the nonlinear effect is negligible: the dynamics is mainly governed by $P^t\bstate{1}$ and the remaining part $\mathcal{R}^{(t)}_\lambda (\bstate{1})$ can be considered as a small residue term.
The proof that  $\mathcal{R}^{(t)}_\lambda(\bstate{1})$ is negligible will be given in \cref{sec: Rt}.

\subsubsection{Assumptions and main result}\label{sec: main result}
In our analysis, in \cref{prop: bounding x1}, we will essentially prove that after one step of the dynamics, $\bstate{1}$ is already close to the limiting value of $\bstate{t}$, with high probability, but did not converge yet.
To this extent, we introduce the parameter $\varepsilon$:
\begin{equation}\label{def: epsilon delta}
    \varepsilon = \varepsilon_n := 
    C \sqrt{\log n} \cdot \max_{i \in V} \sqrt{ \sum_{j\in V} \left(\frac{ W_{ij}}{D_{ii}}\right)^2 },
\end{equation}
where $C$ is such that the bound holds with probability $1-2n^{-\frac{C^2}{18}+1}$ and, hence, with high probability for $C>\sqrt{18}$.
In the next lemma, we prove a lower bound on $\varepsilon$ which will be useful in the proof of the main theorem.

\begin{lemma}\label{lemma: lw bound epsilon}
It holds that
\[
    \varepsilon \ge C \sqrt{\frac{\log n}{n}}.
\]
\end{lemma}
\begin{proof}
Recall that for all \nodes $i\in V$
\(
    \sum_{j\in V} \frac{W_{ij}}{D_{ii}} = 1
\),
and for all $i,j\in V$
\(
    \frac{W_{ij}}{D_{ii}} \geq 0.
\)
Then, for every fixed $i\in V$, by Cauchy-Schwarz inequality we get
\begin{equation*}
    1 = \left(\sum_{j\in V} \frac{W_{ij}}{D_{ii}} \right)^{2}
    =\left(\sum_{j\in V} 1 \cdot \frac{W_{ij}}{D_{ii}} \right)^{2}
    \leq\left(\sum_{j\in V} 1^{2}\right) \cdot \left(\sum_{j=1}^{n} \left(\frac{W_{ij}}{D_{ii}}\right)^{2}\right)
    = n\sum_{j\in V} \left(\frac{W_{ij}}{D_{ii}}\right)^{2}
    \implies
    \sum_{j\in V} \left(\frac{W_{ij}}{D_{ii}}\right)^{2} \geq \frac{1}{n}.
\end{equation*}
We conclude the proof by plugging the above equation into \cref{def: epsilon delta}.
\end{proof}

Recall that we suppose the initial state of our hypergraph to be a $p$-Rademacher random vector for some $p\in [0,1]$ (see \cref{def: p rademacher vector}). Below we collect the assumptions we need to prove our main result. Note that in the case of an initial state being a Rademacher random vector, i.e., $p=1/2$, we need additional assumptions to prove \cref{theorem: main}.

\begin{assumption}[General assumptions on hypergraph]\label{assumptions on Gamma}
Let $\Gamma = (V, H)$ be a hypergraph with $|V|= n$ and let $G(\Gamma)$ be its associated motif graph.
Recall the definitions of $\nbr{i},\nu,\Delta$ respectively given in \cref{eq: def Ni,eq: def nu,eq: def delta}.
We assume that $\Gamma$ is such that:
\begin{enumerate}
    \item the set $N_i$ satisfies $|\nbr{i}| = \omega(\log n),\, \forall i \in V$;
    \item the motif graph $G(\Gamma)$ is connected and non-bipartite, i.e., $\nu<1$;
    \item there exists $m\in \mathbb{N}$ such that 
    \[
        \frac{\nu^m}{1-\nu} \sqrt{\Delta n} \leq m \varepsilon= o(1).
    \]
\end{enumerate}
\end{assumption}

\begin{assumption}[Extra assumptions on hypergraph for $p= 1/2$]\label{extra assumptions on Gamma}
Under the same hypotheses as in \cref{assumptions on Gamma}, if $\bstate{0}$ is a Rademacher random vector, we assume that the motif graph $G(\Gamma)$ is such that:
\begin{enumerate}
    \item there exist $m\in \mathbb{N}$ and $\delta=\delta_n \in (0,1)$ such that 
    \begin{equation*}
        m\varepsilon^2\ll \delta < \varepsilon,
        \quad\text{and }
        \displaystyle \delta = o\left( \frac{\sqrt{\sum_{i\in V} D_{ii}^2}}{\sum_{i\in V} D_{ii}} \right);
    \end{equation*}
    \item it holds \[
        \displaystyle\left( \sum_{i \in V} D_{ii}^2 \right)^{-3/2} \cdot \sum_{i \in V} D_{ii}^3 = o(1).
    \]
\end{enumerate}
\end{assumption}

Note that there exists at least a family of graphs satisfying the above assumptions, as described in the following proposition (whose proof is deferred to \cref{sec: erdos renyi hypergraph}).
\begin{remark}
There exist hypergraphs that satisfy \cref{assumptions on Gamma} and \cref{extra assumptions on Gamma}. 
In particular, as we prove in \cref{prop: erh}, Erd\H{o}s-R\'enyi hypergraphs satisfy the assumptions, with high probability.
\end{remark}

\begin{remark}
When $p=1/2$, we make extra assumptions on the hypergraph since the analysis is harder and requires an additional argument. 
In particular, the simpler argument we used to lower bound the leading contribution of the dynamics in the case $p\neq1/2$ does not hold when $p=1/2$. The technical details will be made clear in the proof of \cref{theorem: main}.
\end{remark}

\begin{theorem}[Convergence]\label{theorem: main}
Let $\lambda \in (-1/2,1/2)$ and $p \in [0,1]$. 
Let $\Gamma=(V,H)$ be a 3-uniform hypergraph with $|V|=n$ such that \cref{assumptions on Gamma} is satisfied. In addition, if $p=1/2$, we also assume $\Gamma$ to satisfy \cref{extra assumptions on Gamma}. Let us consider the dynamics in \cref{eq:dynamics clean} with $\bstate{0}$ being a $p$-Rademacher random vector (see \cref{def: p rademacher vector}). With high probability, for all times $t\ge T$ with
\(
    T = O\left( \frac{\log n}{\log(1/\nu)} \right),
\)
one has
\[
    \bstate{t} = \bar\mu \left(1 + \frac{2p(1-p)(1-{s(2\lambda)})}{1-2p(1-p)(1-{s(2\lambda)})} \pm o(1) \right) \cdot \mathbf{1},
\]
where $\nu,\bar\mu$ are defined as in \cref{eq: def nu,eq: def mu}, respectively.
\end{theorem}
\begin{proof}
By \cref{eq: dynamics final form} and \cref{lemma: computation of pt x1} we have that
\begin{equation}\label{eq: main theorem x t+1}
    \bstate{t+1} = \bar\mu\,\big(1 + \sigma_{\lambda}(\bstate{0})\big) \cdot\mathbf{1} + \mathbf{r}^{(t)} + \mathcal{R}_\lambda^{(t)}(\bstate{1}),   
\end{equation}
where $\bar\mu$ is defined in \cref{eq: def mu}, 
$\sigma_{\lambda}(\bstate{0})$ is defined in \cref{eq: sigmalambda},
$\mathbf{r}^{(t)}$ is such that $\| \mathbf{r}^{(t)} \|_\infty \leq \nu^t \sqrt{\Delta n}$ (with $\nu$ and $\Delta$ as defined in \cref{eq: def nu,eq: def delta}),
and $\mathcal{R}_\lambda^{(t)}$ is the operator defined in \cref{eq:matrix Rt}.

In the remainder of the proof, we will show that $\mathcal{R}_\lambda^{(t)}(\bstate{1})$ and $\mathbf{r}^{(t)}$, which we call \textit{residual terms} (respectively \textit{nonlinear} and \textit{linear}), are negligible with respect to the leading term $\bar\mu\,\big(1 + \sigma_{\lambda}(\bstate{0})\big) \cdot\mathbf{1}$, for large enough $t$.

\paragraph{Leading term}
Let us start by considering the leading term of the dynamics.
Note that:
\begin{enumerate}
    \item[(i)] By \cref{lemma: concentration sigma} (\cref{eq: conc sigma}), which we can apply because of \cref{assumptions on Gamma}.1, we have that, with high probability,   
    \[
        \sigma_{\lambda}(\bstate{0}) = \frac{2p(1-p)(1-{s(2\lambda)})}{1-2p(1-p)(1-{s(2\lambda)})} \cdot (1\pm o(1)) = \Omega(1);
    \]
    \item[(ii)] Moreover, we have $1+\sigma_{\lambda}(\bstate{0})=\Omega(1)$, which holds since $|\sigma_{\lambda}(\bstate{0})|<1$, and $\sigma_{\lambda}(\bstate{0})$ does not depend on $n$.
\end{enumerate}
Hence, we have that
\[
    \bar\mu\,\big(1 + \sigma_{\lambda}(\bstate{0})\big) 
    = \bar\mu\,\left(1 + \frac{2p(1-p)(1-{s(2\lambda)})}{1-2p(1-p)(1-{s(2\lambda)})} \pm o(1) \right)
    = \Theta(\bar\mu).
\]

\paragraph{Nonlinear residual term}
Let us now analyze the residual term due to the nonlinear contribution, i.e., $\mathcal{R}_\lambda^{(t)}(\bstate{1})$.
Since $\Gamma$ satisfies \cref{assumptions on Gamma}.1, we can use \cref{prop: bounding x1} with $\gamma = \varepsilon/6$, and get that, with high probability,
\[
    \|\, \bstate{1} - \mu(1+\sigma_\lambda(\bstate{0})) \cdot\mathbf{1}\, \|_\infty \leq \varepsilon,
\] 
for $\mu:=2p-1$. Note that, by \cref{assumptions on Gamma}.3, $\varepsilon$ is such that $\varepsilon>0$ and $\varepsilon=o(1)$ (see \cref{def: epsilon delta}).
Therefore, \cref{lemma: estimate of R(x)} with $\eta=\varepsilon$, gives us that 
\[
    \|\mathcal{R}_\lambda^{(t)} (\bstate{1})\|_\infty < \frac{8C^2|\lambda|M (|\mu|+\varepsilon)\varepsilon}{1-4C|\lambda|\varepsilon M},
\]
where $C\ge 2$ is a positive constant independent of $n$ and
\[
M = M_\varepsilon := \min\left\{
        m \in \mathbb{N} : \frac{\nu^m}{1-\nu} \sqrt{\Delta n} \leq m \varepsilon
    \right\}.
\]
Note that $M$ is well defined due to \cref{assumptions on Gamma}.3.
Then, because \cref{assumptions on Gamma}.3 implies $M\varepsilon = o(1)$, it follows that 
\begin{equation}\label{eq: bound nonlinear contribution}
    \|\mathcal{R}_\lambda^{(t)} (\bstate{1})\|_\infty
    = O(M(|\mu|+\varepsilon)\varepsilon).
\end{equation}

In order to analyze such an error, we need to distinguish the cases $p \neq 1/2$ and $p = 1/2$. In fact, $\bar{\mu}$ assumes different values in the two cases and we will use different arguments to prove that $\|\mathcal{R}_\lambda^{(t)} (\bstate{1})\|_\infty = o(\bar\mu)$.

\noindent
\underline{Case $p \neq 1/2$}.
In this case, it is enough to note that, with high probability, the leading term $\bar\mu$ is such that $ \bar\mu = \Omega(1)$. This is a consequence of the fact that $|\mu-\bar\mu| < \varepsilon$, which is proven in \cref{lemma:bound on mu} (\cref{eq: conc mu}) using $\gamma=\varepsilon/2$. 
Note that we can apply \cref{lemma:bound on mu} since $\varepsilon=o(1)$ by \cref{assumptions on Gamma}.3.
Therefore,
\[
    \|\mathcal{R}_\lambda^{(t)} (\bstate{1})\|_\infty = O(M(|\mu| +\varepsilon) \varepsilon) = O(M\varepsilon) = o(1),
\]
where we used that $|\mu| = O(1)$ and \cref{assumptions on Gamma}.3. 

\noindent
\underline{Case $p=1/2$}. If $p=1/2$, we have $\mu=2p-1=0$. 
This fact has two consequences, which makes this case different from the case $p\neq 1/2$: 
\begin{enumerate}
\item[(i)] \cref{lemma:bound on mu} only gives us an upper bound, namely $|\bar{\mu}| \le \varepsilon$, with high probability, which is not sufficient since we need a lower bound on $|\bar{\mu}|$ that guarantees it is sufficiently far away from 0.
\item[(ii)] The error deriving from the nonlinear contribution is smaller, since $\varepsilon=o(1)$, namely $O(M(|\mu|+\varepsilon)\varepsilon) = O(M\varepsilon^2)$.
\end{enumerate}
To overcome this problem, we need the additional assumptions given in \cref{extra assumptions on Gamma} that allows us to prove an anti-concentration result for $\bar\mu$.
In particular, we proceed as follows. 
First, we consider a $\delta$ which satisfies \cref{extra assumptions on Gamma}.1.
Second, since we have \cref{extra assumptions on Gamma}.1 and \cref{extra assumptions on Gamma}.2, we apply \cref{thm:lower bound mu} with $a=\delta$ and get that
\begin{equation}\label{eq: barmu lower bound}
    \Pr\left( |\bar\mu| > \delta \right) = 1-o(1).
\end{equation}
Third, we note that, by \cref{extra assumptions on Gamma}.1, $O(M(|\mu| +\varepsilon) \varepsilon) =  O(M\varepsilon^2) = o(\delta)$.
Hence, from the above equation, it follows that $\|\mathcal{R}_\lambda^{(t)} (\bstate{1})\|_\infty = o(\bar{\mu})$, with high probability, for every $t>1$.

\paragraph{Linear residual term}
Let us move to the residual term given by the linear contribution, i.e., $\mathbf{r}^{(t)}$. 
In the following, we prove that $\| \mathbf{r}^{(t)} \|_\infty  = o(\bar\mu)$, for large enough $t$.
Also here we distinguish the two cases $p\neq 1/2$ and $p=1/2$.

\noindent
\underline{Case $p \neq 1/2$}.
With high probability, by \cref{lemma:bound on mu} (\cref{eq: conc mu}) with $\gamma=\varepsilon/2$, we have that $\bar\mu > \mu - \varepsilon$, which implies, since $\varepsilon=o(1)$, that $\bar\mu > c\mu$ for some constant $c\in(0,1)$.

\noindent
\underline{Case $p = 1/2$}.
With high probability, as proved in \cref{eq: barmu lower bound}, we have that $\bar\mu > \delta$, for $\delta$ that satisfies \cref{extra assumptions on Gamma}.1.

Putting together the two cases $p \neq 1/2$ and $p = 1/2$, it follows that, with high probability, the leading term $\bar\mu > \min(c\mu, \delta) = \delta$, because $c\mu=\Omega(1)$ and $\delta=o(1)$.
Recall from \cref{lemma: computation of pt x1} that $\| \mathbf{r}^{(t)} \|_\infty \leq \nu^t \sqrt{\Delta n}$ and that, by \cref{assumptions on Gamma}.2, $|\nu|<1$.
Note that, if $\nu^t \sqrt{\Delta n} = o(\delta)$ then $\| \mathbf{r}^{(t)} \|_\infty = o(\delta) = o(\bar\mu)$.
Let us now define $T$ as 
\[
    T := \frac{\log\left( \frac{\sqrt{\Delta n}}{o(\delta)} \right)}{\log(1/\nu)}.
\]
Then, for all times $t\ge T$, we have that $\| \mathbf{r}^{(t)} \|_\infty \le \nu^T \sqrt{\Delta n} = o(\delta)$, namely $\| \mathbf{r}^{(t)} \|_\infty = o(\bar\mu)$, for every $p \in[0,1]$, for all times $t\ge T$. 

Finally note that, by \cref{eq:graph W,eq:graph D,eq: def delta}, it holds that $\Delta \le n^2$. 
Moreover, from \cref{extra assumptions on Gamma}.1, it follows that $\delta \gg \varepsilon^2$ and, by \cref{lemma: lw bound epsilon}, it follows that $\varepsilon = \Omega(\sqrt{\log(n)/n})$. 
Therefore, we conclude that the dynamics converges within time $T\le\frac{\log(n^{5/2} / \log n)}{\log(1/\nu)} =O\left(\frac{\log n}{\log(1/\nu)}\right)$.
\end{proof}

\begin{remark}[Shift depends on $p$, $s$ and $\lambda$]
\cref{theorem: main} essentially claims that, under certain conditions on the hypergraph and when the dynamics starts from a $p$-Rademacher random initialization, it converges with high probability to a configuration where all \nodes have state $\bar\mu \left(1 + \frac{2p(1-p)(1-{s(2\lambda)})}{1-2p(1-p)(1-{s(2\lambda)})} \pm o(1)\right)$, namely to a \textit{shifted} weighted average of the initial states. 
In fact, here $\bar\mu$ is the weighted average of the initial states (what the linear dynamics converges to), and $\frac{2p(1-p)(1-{s(2\lambda)})}{1-2p(1-p)(1-{s(2\lambda)})}$ is a multiplicative factor depending on $p$, $s$ and $\lambda$ that shifts the average due to the nonlinear interaction between the \nodes.
\end{remark}

\begin{remark}[Shift for $s(x)=e^{x}$]\label{eq: shift NML20}
{
Note that if $s(x)=e^{x}$, as modeled in~\cite{PhysRevE.101.032310}, we have $\frac{2p(1-p)(1-e^{2\lambda})}{1-2p(1-p)(1-e^{2\lambda})} \in (-1,1)$ for every $p$ and $\lambda$ and, in particular, that for every $p\in(0,1)$:
\begin{equation}
\frac{2p(1-p)(1-e^{2\lambda})}{1-2p(1-p)(1-e^{2\lambda})}
\begin{cases}
   > 0 & \text{if } \lambda<0,
   \\
   = 0 & \text{if } \lambda=0,
   \\
   < 0 & \text{if } \lambda>0.
\end{cases}
\end{equation}
In other words, the nonlinear effect shifts the state of convergence either toward the initial majority, when $\lambda<0$, or toward 0, when $\lambda>0$. The nonlinear effect is canceled when $\lambda=0$ given that the dynamics reduces to a linear one (as proven in \cref{prop: linear interaction}).
The intensity of the shift, instead, is modeled by the parameter $|\lambda|$ (the bigger $|\lambda|$, the stronger the shift) and by the balance of states in the initial configuration (the smaller $|1/2-p|$, the stronger the shift).
We remark that this is the same effect we see in the mean-field model (\cref{prop: trivial topology}) with $a=pn$ and $b=(1-p)n$, and qualitatively the same that has been shown empirically in~\cite{PhysRevE.101.032310}.}
\end{remark}

\section{On the nonlinear effect}\label{sec: nonlinear effect}
In this section, we prove that the effect of the nonlinear dynamics is negligible for any $t\geq 1$. More precisely, recall that for any fixed $\lambda\in (-1/2,1/2)$, we can write the dynamics as 
\[ \bstate{t+1} = \mathcal{P}^t(\bstate{1}) + \mathcal{R}^{(t)}_\lambda (\bstate{1}) = P\bstate{1} + \mathcal{R}^{(t)}_\lambda (\bstate{1}),
\]
where $\mathcal{P}$ is defined in \cref{eq: op P acts as a matrix} and $\mathcal{R}^{(t)}_\lambda$ is as in \cref{eq:matrix Rt}. We will analyze the two contributions $P\bstate{1}$ and $\mathcal{R}_\lambda(\bstate{1})$ separately. 
In particular, we will fully characterize the first term $P^t\bstate{1}$ as the value the dynamics converges to, containing the effect of the higher-order nonlinear interactions, plus a linear residual term.
Moreover, we will prove that the residual nonlinear contribution $\mathcal{R}_\lambda(\bstate{1})$ is negligible.

\subsection{On the leading term \texorpdfstring{$P^t\bstate{1}$}{}}
To start the analysis of ${P}^t\bstate{1}$, in the next proposition we describe how $\bstate{1}$ depends on the initial state $\bstate{0}$, separating the linear effect in the evolution (i.e., the first order approximation in \cref{pro: 2nd order taylor}) from the nonlinear effect.

\begin{lemma}[The state $\bstate{1}$]\label{lemma: value of xi(1)}
Let $\bstate{0} \in \{-1,+1\}^n$. 
Let $\bstate{1}$ evolve according to \cref{eq:dynamics clean} for $\lambda\in (-1/2,1/2)$, i.e., for every $i$ we have
\[
    \state{1}_i = \frac{1}{\snorm{0}_i} \sum_{ \{j,k\} \in \nbr{i}} {s(\lambda|\state{0}_j-\state{0}_k|)} \cdot \frac{\state{0}_j+\state{0}_k}{2},
\]
where $\snorm{0}_i$ is the normalization factor defined in \cref{eq:normalization}.
For every $i\in V$, define
\begin{equation}\label{eq: def ci}
    c_i := \big| \{ \{j,k\} \in \nbr{i} : \state{0}_j \neq \state{0}_k \} \big|.
\end{equation}
For every $i\in V$, it holds
\[
    \state{1}_i = \bar\mu_i\cdot \left( 1 + \sigma_{\lambda;i}(\bstate{0}) \right),
\]
with
\begin{equation}\label{eq: def mui sigmai}
    \bar\mu_i := \frac{1}{D_{ii}}\sum_{j\in V} W_{ij}\,\state{0}_j\quad
    \quad\text{and}\quad
    \sigma_{\lambda;i}(\bstate{0}) := \frac{c_i\,(1-{s(2\lambda)})}{|\nbr{i}| -c_i\,(1-{s(2\lambda)})},
\end{equation}
where $W=(W_{ij})$ is the adjacency matrix defined in \cref{eq:graph W} and $D=(D_{ii})$ us the diagonal degree matrix defined in \cref{eq:graph D}.
\end{lemma}
\begin{proof}
Instead of using Taylor expansion to approximate $\bstate{1}$ via the linear dynamics (as done in \cref{pro: 2nd order taylor}), we consider the full contribution of the nonlinear dynamics. 
In fact, as for the mean-field model (see \cref{prop: trivial topology}), we can exhaustively%
\footnote{Differently than in the mean-field case, the dynamics on a general hypergraph does not converge in one single step.} analyze $\bstate{1}$.
Note that, since $\bstate{0} \in \{-1,+1\}^n$, for every $i\in V$ it holds that
\begin{align*}
    \state{1}_{i} 
    &= \frac{1}{\snorm{0}_i} \left(
        \sum_{\substack{\{j,k\}\in \nbr{i}\\\state{0}_j=\state{0}_k}} {s(\lambda|\state{0}_j-\state{0}_k|)} \,  \frac{\state{0}_{j}+\state{0}_{k}}{2}
        +\sum_{\substack{\{j,k\}\in \nbr{i}\\\state{0}_j\neq\state{0}_k}} {s(\lambda|\state{0}_j-\state{0}_k|)} \,  \frac{\state{0}_{j}+\state{0}_{k}}{2}
    \right)
    \\
    &=\frac{1}{\snorm{0}_i} \sum_{\substack{\{j,k\}\in \nbr{i}\\\state{0}_j=\state{0}_k}} \state{0}_j .
\end{align*}
Let us now define the following quantities, for every $i\in V$:
\begin{align*}
\begin{cases}
a_i := |\{ \{j,k\} \in \nbr{i} : \state{0}_j = \state{0}_k = +1 \}|,
\\
b_i := |\{ \{j,k\} \in \nbr{i} : \state{0}_j = \state{0}_k = -1 \}|,
\\
c_i := |\{ \{j,k\} \in \nbr{i} : \state{0}_j \neq \state{0}_k \}|.
\end{cases}
\end{align*}
Note that $a_i+b_i+c_i=|\nbr{i}|$. 
By using the definition of $\sigma_{\lambda;i}(\bstate{0})$ as in \cref{eq: def mui sigmai}, we get
\begin{align*}
    \state{1}_{i} 
    &= \frac{a_i-b_i}{a_i + b_i + c_i\cdot{s(2\lambda)}}
    = \frac{a_i-b_i}{|\nbr{i}| - c_i\,(1-{s(2\lambda)})}
    = \frac{a_i-b_i}{|\nbr{i}|} 
    \cdot \left( 1 + \sigma_{\lambda;i}(\bstate{0}) \right)
    \\
    &= \frac{1}{|\nbr{i}|} \sum_{ \{j,k\} \in \nbr{i} } \frac{(\state{0}_j+\state{0}_k)}{2}
    \cdot \left( 1 + \sigma_{\lambda;i}(\bstate{0}) \right)
    \\
    &= \frac{1}{D_{ii}} \sum_{ j \in V } W_{ij} \state{0}_j
    \cdot \left( 1 + \sigma_{\lambda;i}(\bstate{0}) \right)
    = \bar\mu_i
    \cdot \left( 1 + \sigma_{\lambda;i}(\bstate{0}) \right),
\end{align*}
where, in the last equality, we used that by definition $\bar\mu_i := \frac{1}{D_{ii}} \sum_{ j \in V } W_{ij} \state{0}_j$. This concludes the proof.
\end{proof}

In the next proposition, we analyze the nonlinear effect on the dynamics, which depends on the quantities $\sigma_{\lambda;i}(\bstate{0})$ defined in \cref{eq: def mui sigmai}.
In particular, we show that $P^t\bstate{1}$ converges to a value that is multiplicatively shifted with respect to the convergence value of $P^t\bstate{0}$, as an effect of the nonlinear contribution.

\begin{proposition}[Multiplicative shift as a nonlinear effect]\label{lemma: computation of pt x1}
Let $i\in V$ and let $\bstate{0} \in \{-1,+1\}^n$. For any $\lambda\in (-1/2,1/2)$, let $\sigma_{\lambda;i}(\bstate{0})$ as in \cref{eq: def mui sigmai}. We define
\begin{equation}\label{eq: sigmalambda}
    \sigma_{\lambda}(\bstate{0}) := \frac{\sum_{i \in V} \left( \sum_{j \in V} W_{ij}\,\sigma_{\lambda;j}(\bstate{0})\right) \state{0}_i }{\sum_{i \in V} D_{ii}\,\state{0}_i},
\end{equation}
where $W=(W_{ij})$ is the adjacency matrix defined in \cref{eq:graph W} and $D=(D_{ii})$ is the diagonal degree matrix defined in \cref{eq:graph D}.
It holds that
\[
     P^t\bstate{1} = \bar\mu\,\big(1 + \sigma_{\lambda}(\bstate{0})\big) \cdot\mathbf{1} + \mathbf{r}^{(t)},
    \quad\text{with }
    \| \mathbf{r}^{(t)} \|_\infty \leq \nu^t \sqrt{\Delta n},
\]
where $\bar\mu$ is as in \cref{eq: def mu}, and $\nu$ and $\Delta$ are defined as in \cref{eq: def nu,eq: def delta}.
\end{proposition}
\begin{remark}[Comparison between ${P}\bstate{1}$ and ${P}\bstate{0}$]
Note that the proposition above allows us to compare ${P}\bstate{1}$ (which contains the nonlinear effect in $\bstate{1}$) with the linear dynamics ${P}\bstate{0}$ (see \cref{prop: linear dynamics convergence}), highlighting how the difference between the two evolutions depends on the multiplicative shift $\sigma_{\lambda}(\bstate{0})$, defined in \cref{eq: sigmalambda}. 
\end{remark}
\begin{proof}[Proof of \cref{lemma: computation of pt x1}]
Recall the definition of $c_i := \big| \{ \{j,k\} \in \nbr{i} : \state{0}_j \neq \state{0}_k \} \big|$. By using \cref{lemma: value of xi(1)}, we get that for every $i \in V$
\[
    \state{1}_{i} = \bar\mu_i \cdot \left( 1 + \sigma_{\lambda;i}(\bstate{0}) \right),
\]
with
\[
    \bar\mu_i := \frac{1}{D_{ii}}\sum_{j\in V} W_{ij}\,\state{0}_j
    \quad\text{and}\quad
    \sigma_{\lambda;i}(\bstate{0}) := \frac{c_i\,(1-{s(2\lambda)})}{|\nbr{i}| -c_i\,(1-{s(2\lambda)})}.
\]
By using \cref{prop: linear dynamics convergence} we get that 
\[
   {P}^t(\bstate{1}) = \bar\mu^{(1)} \mathbf{1} + \mathbf{r}^{(t)},
    \quad\text{with }
    \bar\mu^{(1)} := \frac{\sum_{i\in V}D_{ii} x_i^{(1)}}{\sum_{i \in V} D_{ii}}
    \quad\text{and }
    \| \mathbf{r}^{(t)} \|_\infty \leq \nu^t \sqrt{\Delta n}.
\] 
Now our goal is to explicitly write how $\bar\mu^{(1)}$ depends on $\bstate{0}$, instead of $\bstate{1}$ as described by the previous equation. Using now \cref{lemma: value of xi(1)} together with the fact that the weighted adjacency matrix satisfies $W_{ij}=W_{ji}$ for every $i,j\in V$, we rewrite the numerator of $\bar\mu^{(1)}$ as follows
\begin{align*}
    &\sum_{i\in V} D_{ii}\,\state{1}_i
    = \sum_{i \in V} D_{ii} \cdot \frac{1}{D_{ii}}\sum_{j\in V} W_{ij}\,\state{0}_j \cdot \left( 1 + \sigma_{\lambda;i}(\bstate{0}) \right)
    \\
    &= \sum_{j \in V} \state{0}_j \sum_{i \in V} W_{ji}
    \cdot \left( 1 + \sigma_{\lambda;i}(\bstate{0}) \right)
    = \sum_{j \in V} D_{jj}\,\state{0}_j
    \cdot \left( 1 + \frac{1}{D_{jj}} \sum_{i\in V} W_{ji}\, \sigma_{\lambda;i}(\bstate{0}) \right).
\end{align*}
Therefore, we derive that
\[
    \bar\mu^{(1)} 
    = \frac{\sum_{i \in V} D_{ii} x_i^{(0)} (1 + \frac{1}{D_{ii}} \sum_{j \in V} W_{ij}\,\sigma_{\lambda;j}(\bstate{0}))}{\sum_{i \in V} D_{ii}}.
\]
Hence, by defining $\sigma_{\lambda}(\bstate{0})$ as in \cref{eq: sigmalambda}, and for $\bar\mu$ as defined in \cref{eq: def mu}, we conclude that
\[
    {P}^t\bstate{1} = \bar\mu\,\big(1 + \sigma_{\lambda}(\bstate{0})\big) \cdot\mathbf{1} + \mathbf{r}^{(t)}.
\]
\end{proof}

\subsection{On the residual nonlinear term \texorpdfstring{$\mathcal{R}_\lambda^{(t)}(\bstate{1})$}{}}\label{sec: Rt}
The main result of this section is the proof of the fact that $\mathcal{R}_\lambda^{(t)}(\bstate{1})$ is negligible with respect to the nonlinear effect in the expression of $\mathcal{P}(\bstate{1})$, described by $\bar{\mu}\sigma_{\lambda}(\bstate{0})$.
In particular, we will bound the $\ell^\infty$ norm of $\mathcal{R}_\lambda^{(t)}(\bstate{1})$ with respect to the following parameter, that depends on the hypergraph topology (i.e., on $\nu,\Delta,n$):
\begin{equation}\label{def: M}
    M_\xi := \min\left\{
        m \in \mathbb{N} : \frac{\nu^m}{1-\nu} \sqrt{\Delta n} \leq m {\xi}
    \right\}, \qquad \xi>0.
\end{equation}
Note that that $M$ is well defined since, under \cref{assumptions on Gamma}, it always exists $m \in \mathbb{N}$ such that $\frac{\nu^m}{1-\nu} \sqrt{\Delta n} \leq m \varepsilon$.
The rigorous statement is given below.

\begin{proposition}[Nonlinear residual estimation]\label{lemma: estimate of R(x)}
Let ${\eta}>0$ be such that $\eta = o(1)$ and let $M_\eta$ be as in \cref{def: M}, for $\Delta,n,\nu$ such that $M_\eta$ exists.
Let also $\eta$ be such that $\eta M_\eta  = o(1)$. 
For $\rho \in \mathbb{R}$, assume that $|\state{1}_i - \rho| < \eta$, for every $i$. 
Let $\mathcal{R}_\lambda^{(t)}$ as in \cref{eq:matrix Rt}.
It holds that
\[
    \|\mathcal{R}_\lambda^{(t)} (\bstate{1})\|_\infty <
    \frac{8C^2 |\lambda| \eta M_\eta (|\rho|+{\eta})}{1-4C |\lambda| \eta M_\eta},
\]
where the constant $C>2$ is that appearing in \cref{lem: bounds on norms}.
\end{proposition}

Note that \cref{lemma: estimate of R(x)} will be used in the proof of \cref{theorem: main}, where $\rho$ represents the expected value the dynamics converges to, while $\eta$ represents how far the dynamics is from $\rho$ at time $t=1$. The assumptions on $\eta,M_\eta$ are satisfied by \cref{assumptions on Gamma} and \cref{extra assumptions on Gamma}, while we will prove in \cref{prop: bounding x1} that the assumption on $|\state{1}_i - \rho| < \eta$ is satisfied, with high probability, whenever the initial state is a $p$-Rademacher random vector.

Before discussing the proof of \cref{lemma: estimate of R(x)}, we need the following two auxiliary results (\cref{lem: bounds on norms}, \cref{lem: bound on the sum of norms}).
\begin{lemma}\label{lem: bounds on norms}
Let $\mathbf{y} \in \mathbb{R}^n$ such that $\|\mathbf{y}\|_\infty\leq 1$. Let ${P},\mathcal{Q}_\lambda$ be as in \cref{eq:nonlinear dynamics matrix apx}. For every $m \in \mathbb{N}$, there exists a positive finite constant $C>2$ such that
\[
    \|\mathcal{Q}_\lambda  ({P}^m \mathbf{y})\|_\infty \leq
    \begin{cases}
        C \, \displaystyle \max_{i,j\in V}|y_i - y_j| \|\mathbf{y}\|_\infty  \, ,
        \\
        C \, \nu^m \sqrt{\Delta n} \, \|\mathbf{y}\|_\infty \, ,
    \end{cases}
\]
where $\nu$ and $\Delta$ are defined as in \cref{eq: def nu} and \cref{eq: def delta}, respectively.
\end{lemma}
\begin{proof}
For short, we set $\mathbf{z}=P^m \mathbf{y}$. Since $\|P\|_\infty \leq 1$, it follows that $\|\mathbf{z}\|_\infty \le \|\mathbf{y}\|_\infty \leq 1$. Then by \cref{eq:matrix Qt}, we have that
\begin{align}
[\mathcal{Q}_\lambda(\mathbf{z})]_{i} 
&= \frac{1}{\lambda} \frac{1}{D_{ii}} \sum_{j\in V} \left[
    \sum_{k\in V} A_{ijk} \left(
        \frac{\lambda}{|\nbr{i}|} \sum_{\{\ell,r\}\in\nbr{i}} \left( 
            |z_j-z_k| - |z_\ell-z_r| 
        \right) + \mathcal{E}_{\lambda;i,j,k}^{(t)} 
    \right) \right] z_j
\notag\\
&= \sum_{j\in V}\sum_{k\in V} \sum_{\{\ell,r\}\in\nbr{i}}
\frac{A_{ijk}}{|\nbr{i}|\,D_{ii}} \left( \left(
    |z_j-z_k| - |z_\ell-z_r|
\right)
+ \frac{1}{\lambda}\mathcal{E}_{\lambda;i,j,k}^{(t)}\right) z_j
\notag\\
&\leq \left( 2\max_{\{j,k\}\in N_i} |z_j - z_k| + O\Big(\lambda \max_{\{j,k\}\in N_i} |z_j - z_k|^2\Big)\right) \, \|\mathbf{z}\|_\infty
\notag\\
&\leq C\max_{j,k \in V} |z_j - z_k| \|\mathbf{z}\|_\infty,
\label{eq: bound on Q(z)i}
\end{align}
where $C>2$ is a positive constant and where we used \cref{eq:nonlinear dynamics matrix apx,eq:graph D} to get the last inequality.
In particular, note that
$\max_{j\in V}\,[P \mathbf{y}]_j
=\max_{j\in V} \sum_{\ell\in V} \frac{W_{j\ell}}{D_{jj}} y_\ell
\le \max_{\ell\in V} y_\ell$ 
since $\sum_{\ell\in V} \frac{W_{j\ell}}{D_{jj}}=1$, and similarly
$\min_{j\in V}\,[P \mathbf{y}]_j \ge \min_{\ell\in V} y_\ell$.
The same holds also for $P^m$ since the matrix is still row-stochastic.
Therefore:
\begin{equation}\label{eq: first upper bound on the difference of pky}
    \max_{j,k \in V} \left|[P^m \mathbf{y}]_j - [P^m \mathbf{y}]_k\right| 
    = \max_{j\in V}\,[P^m \mathbf{y}]_j - \min_{k\in V}\,[P^m \mathbf{y}]_k 
    \leq \max_{j\in V} y_j - \min_{k\in V} y_k
    = \max_{j,k \in V}|y_j - y_k|.
\end{equation}
Therefore, by using that $\mathbf{z}=P^m\mathbf{y}$, that $\|\mathbf{z}\|_\infty \leq \|\mathbf{y}\|_\infty$, and by applying \cref{eq: first upper bound on the difference of pky} into \cref{eq: bound on Q(z)i} it follows that
\[
    \|\mathcal{Q}_\lambda(P^m \mathbf{y} )\|_\infty 
    \leq C \max_{j,k \in V} \left|[P^m \mathbf{y}]_j - [P^m \mathbf{y}]_k\right|\|\mathbf{y}\|_\infty 
    \leq C \, \max_{j,k\in V}|y_j - y_k|\, \|\mathbf{y}\|_\infty\, .
\]
For the second upper bound, note that \cref{prop: linear dynamics convergence} directly implies the following bound:
\begin{equation}\label{eq: second upper bound on the difference of pky}
    \|P^m \mathbf{y} - \bar\mu \mathbf{1} \|_\infty \leq \nu^m \sqrt{\Delta n},
\end{equation}
with $\nu,\Delta,\bar\mu$ as defined in \cref{eq: def mu,eq: def delta,eq: def nu}.
Hence, for every $j,k \in V$ it holds that
\[
    \left|[P^m \mathbf{y}]_j - [P^m \mathbf{y}]_k\right| 
    \leq 2\nu^m \sqrt{\Delta n}
\]
and therefore by using \cref{eq: bound on Q(z)i,eq: first upper bound on the difference of pky,eq: second upper bound on the difference of pky} we get 
\begin{align*}
    \|\mathcal{Q}_\lambda(P^m \mathbf{y})\|_\infty 
   \leq \, C\nu^m\sqrt{\Delta n} \|\mathbf{y}\|_\infty \, .
\end{align*}
This gives us the desired lemma.
\end{proof}

\begin{lemma}\label{lem: bound on the sum of norms}
Let ${\eta}>0$ and $M_\eta$ be as in \cref{def: M}, for $\Delta,n,\nu$ such that $M_\eta$ exists.
Let $\mathbf{y} \in \mathbb{R}^n$ such that $\|\mathbf{y}\|_\infty \le 1$ and $\max_{i,j\in V}|y_i - y_j| < {\eta}$. Let $P,\mathcal{Q}_\lambda$ be as in \cref{eq:nonlinear dynamics matrix apx} for any $\lambda\in (-1/2,1/2)$. Then we have 
\[
    \sum_{m=0}^{\infty} \| \mathcal{Q}_\lambda(P^m \mathbf{y}) \|_\infty \leq 2C\eta M_\eta \|\mathbf{y} \|_\infty,
\]
where $C>2$ is the positive constant appearing in \cref{lem: bounds on norms}.
\end{lemma}
\begin{proof}
The proof follows from \cref{lem: bounds on norms} together with \cref{assumptions on Gamma}. More precisely, let $M_\eta$ be as in \cref{def: M}. We write
\begin{align*}
\sum_{m=0}^{\infty} \| \mathcal{Q}_\lambda(P^m\mathbf{y}) \|_\infty
&= \sum_{m=0}^{M_\eta-1} \| \mathcal{Q}_\lambda (P^m\mathbf{y}) \|_\infty + \sum_{m=M_\eta}^{\infty} \| \mathcal{Q}_\lambda(P^m\mathbf{y}) \|_\infty 
\\
&\leq C M_\eta \, \max_{i,j\in V}|y_i - y_j|\,\|\mathbf{y} \|_\infty \,  +  C\sum_{m=M_\eta}^{\infty}  \, \nu^m \sqrt{\Delta N} \, \|\mathbf{y}\|_\infty,
\end{align*}
where we used the two bounds on $\|\mathcal{Q}_\lambda(P^m\mathbf{y})\|_\infty$ proved in \cref{lem: bounds on norms} for some $C>2$. Therefore, being $|\nu|<1$ and $\max_{i,j\in V}|y_i - y_j| <\eta$ by assumption, we have
\[
    \sum_{m=0}^{\infty} \| \mathcal{Q}_\lambda (P^m\mathbf{y}) \|_\infty \leq C \Big(\eta M_\eta + \frac{\nu^{M_\eta}}{1-\nu} \sqrt{\Delta N} \Big) \, \|\mathbf{y}\|_\infty.
\]
To conclude the proof it is enough to notice that from the definition of $M_\eta$ (see \cref{def: M}), we have that
$\frac{\nu^M}{1-\nu} \sqrt{\Delta N} \leq \eta M_\eta$, which guarantees that 
\[
     \sum_{m=0}^{\infty} \| \mathcal{Q}_\lambda (P^m\mathbf{y}) \|_\infty\leq 2C \eta M_{\eta} \|\mathbf{y}\|_\infty,
\]
where $C>2$ is the constant appearing in \cref{lem: bounds on norms}.
This concludes the proof.
\end{proof}

By applying \cref{lem: bounds on norms,lem: bound on the sum of norms}, we can now prove \cref{lemma: estimate of R(x)}.
\begin{proof}[Proof of \cref{lemma: estimate of R(x)}]
In this proof, we simplify the notation and we write $\mathcal{R}^{(t)}$ for the operator $\mathcal{R}_\lambda^{(t)}$ defined in \cref{eq:matrix Rt}. 
Recall from \cref{eq:matrix Rt} that
\begin{equation}\label{eq: def Rk}
    \mathcal{R}^{(t)} 
    = (\mathcal{P}+\lambda \mathcal{Q}_\lambda)^t-\mathcal{P}^t
    = \sum_{k=1}^{t} \lambda^k \mathcal{R}_{k}^{(t)},
\end{equation}
where we denote with $\mathcal{R}_{k}^{(t)}$ the sum of all possible composite operators that contain $(t-k)$ times the operator $\mathcal{P}$ and $k$ times the operator $\mathcal{Q}_\lambda$. 

The core of the proof is proving the following bound, for every $k\in \mathbb{N}$:
\begin{equation}\label{eq: bound on Rk x}
    \|\mathcal{R}_{k}^{(t)} (\bstate{1})\|_\infty \leq 2C(4C\eta M_\eta)^{k}\|\bstate{1}\|_\infty,
\end{equation}
where $C>2$ is the positive and finite constant introduced in \cref{lem: bounds on norms}.
This is indeed enough to conclude the proof since it will allow us to bound uniformly with respect to $t$ the sum in the right-hand side of \cref{eq: def Rk}, as it is going to be clear later. 

\medskip
In order to prove \cref{eq: bound on Rk x}, let us start with $k=1$. Note that 
\begin{align}
    \| \mathcal{R}_1^{(t)} (\bstate{1}) \|_\infty
    &=\left\| \sum_{m=0}^{t-1} P^{t-m-1} \mathcal{Q}_\lambda(P^m\bstate{1}) \right\|_\infty \leq \sum_{m=0}^{t-1} \| P^{t-m-1} \mathcal{Q}_\lambda(P^m\bstate{1}) \|_\infty 
    \notag\\
    &\leq \sum_{m=0}^{t-1} \| P^{t-m-1} \|_\infty \| \mathcal{Q}_\lambda(P^m\bstate{1}) \|_\infty \leq \sum_{m=0}^{\infty} \| \mathcal{Q}_\lambda(P^m\bstate{1}) \|_\infty,
    \label{eq: applying lemma to get bound on R1}
\end{align}
where the last inequality follows since $\|P\|_\infty\le 1$.
Note that, since 
$|\state{1}_i - \rho| < \eta$ for every $i$, 
it holds that $\max_{i,j\in V}|\state{1}_i - \state{1}_j|<2\eta$.
Therefore, by applying \cref{lem: bound on the sum of norms} with $\mathbf{y} = \bstate{1}$, we immediately get that there exists a positive and finite constant $C>2$ such that 
\begin{equation}\label{eq: R1t}
    \| \mathcal{R}_1^{(t)} (\bstate{1}) \|_\infty \leq 2C(2\eta) M_{2\eta} \|\bstate{1}\|_\infty \leq 4C\eta M_{\eta} \|\bstate{1}\|_\infty,
\end{equation}
where the last inequality directly follows from the definition of $M_{2\eta}$ and $M_\eta$ (see \cref{def: M}).

\medskip
For $k=2$, we proceed analogously to \cref{eq: applying lemma to get bound on R1} and get:
\begin{align}\label{eq:QPellzm}
\|\mathcal{R}_{2}^{(t)} (\bstate{1})\|_\infty 
&= \left\| \sum_{m=0}^{t-2} \sum_{r=0}^{t-m-2} P^{t-m-r-2} \mathcal{Q}_\lambda \big( P^{r} \mathcal{Q}_\lambda (P^{m}\bstate{1})\big) \right\|_\infty 
\\\notag
&\leq \sum_{m=0}^{\infty} \left(\sum_{r=0}^{\infty} \| \mathcal{Q}_\lambda \big( P^{r} \mathcal{Q}_\lambda (P^{m}\bstate{1})\big) \|_\infty \right).
\end{align}
We want now to apply \cref{lem: bound on the sum of norms} to \cref{eq:QPellzm} for each fixed $m$ with $\mathbf{y} = \mathcal{Q}_\lambda (P^{m}\bstate{1})$. Note that this is possible since, by \cref{lem: bounds on norms} and the fact that $\max_{i,j\in V} |x_i^{(1)} - x_j^{(1)}| \leq 2\eta$ (which follows from the assumption $|\state{1}_i - \rho| < \eta$ for every $i$), it holds that 
\begin{equation}\label{eq:infnorm bound y}
    \|\mathbf{y}\|_\infty=
    \|\mathcal{Q}_\lambda(P^{m}\bstate{1})\|_\infty \le C \max_{i,j\in V}|\state{1}_i - \state{1}_j|\,\|\bstate{1}\|_\infty < 2C\eta
\end{equation}
where the constant $C$ introduced in \cref{lem: bounds on norms} is finite and such that $C>2$ and where we used that $\|\bstate{1}\|_\infty \leq 1$ (see \cref{rem: infty norm x t}).
Therefore, we get 
\begin{equation}\label{eq:difference bound y}
    \max_{i,j\in V} |y_i - y_j | < 4C \eta,
\end{equation}
for some $C>2$.
Now, from \cref{lem: bound on the sum of norms}, with $\mathbf{y} = \mathcal{Q}_\lambda (P^{m}\bstate{1})$ and $\widetilde{\eta} = 4C \eta$, we find that 
\begin{equation}\label{eq: bound sum r}
    \sum_{r=0}^\infty \|\mathcal{Q}_\lambda(P^r\mathbf{y})\| \leq 2C \,\widetilde{\eta} \,M_{\widetilde{\eta}} \,\|\mathbf{y}\|_\infty \leq 8 C^2 \eta M_{\widetilde{\eta}} \|\mathbf{y}\|_\infty.
\end{equation}
Recall that for any $\xi>0$, $M_{\xi}$ is defined as (see \cref{def: M})
\[
    M_{\xi} := \min\left\{
        m \in \mathbb{N} : \frac{\nu^m}{1-\nu} \sqrt{\Delta n} \leq m\, \xi
    \right\}.
\]
Therefore, being $\widetilde{\eta} > \eta$, it holds that $M_{\widetilde{\eta}} \le M_\eta$. Thus, by plugging \cref{eq: bound sum r} into \cref{eq:QPellzm} and using $M_{\widetilde{\eta}} \le M_\eta$, we get that 
\begin{equation}\label{eq: R2 case1}
    \|\mathcal{R}_{2}^{(t)} (\bstate{1}) \|_\infty \leq 8 C^2\,{\eta}\, M_\eta \sum_{m=0}^{\infty} \| \mathcal{Q}_\lambda(P^{m}\bstate{1})\|_\infty.
\end{equation}
Using again \cref{lem: bound on the sum of norms} to bound the series with respect to the parameter $m$, we get 
\[
    \sum_{m=0}^\infty\|\mathcal{Q}_\lambda(P^m \bstate{1}) \|_\infty \leq 4C \eta M_\eta \|\bstate{1}\|_\infty,
\]
where we proceeded as in \cref{eq: R1t}.
Therefore, by combining \cref{eq: R2 case1} with the estimate above, we conclude that  
\begin{equation}\label{eq: est R2}
    \|\mathcal{R}_{2}^{(t)} (\bstate{1})\|_\infty 
    \leq 32 C (C\eta M_\eta)^2
        \|\bstate{1}\|_\infty = 2C (4C \eta M_\eta)^2 \|\bstate{1}\|_\infty
\end{equation}

\medskip
For $k=3$, proceeding as in \cref{eq: applying lemma to get bound on R1}, we get:
\begin{align*}
\|\mathcal{R}_{3}^{(t)} (\bstate{1})\|_\infty
&= \left\| 
\sum_{m=0}^{t-3}
    \sum_{r=0}^{t-m-3}
        \sum_{s=0}^{t-m-r-3} 
        P^{t-m-r-s-3} \mathcal{Q}_\lambda \big( P^s \mathcal{Q}_\lambda (P^r \mathcal{Q}_\lambda (P^m\bstate{1})) \big)
\right\|_\infty
\\
&\leq \sum_{m=0}^{\infty} \left(
    \sum_{r=0}^{\infty} \left(
        \sum_{s=0}^{\infty} 
        \| 
        \mathcal{Q}_\lambda \big( P^s \mathcal{Q}_\lambda (P^r \mathcal{Q}_\lambda (P^m\bstate{1})) \big)
        \|_\infty
    \right)
\right).
\end{align*}
Similarly as for $k=2$,
we define
$\mathbf{y}=\mathcal{Q}_\lambda(P^m\bstate{1})$ and $\mathbf{z}=\mathcal{Q}_\lambda(P^r\textbf{y})$.
We already know the upper bounds on $\mathbf{y}$ from 
\cref{eq:infnorm bound y,eq:difference bound y}. In particular, we know that 
\[
    \|\mathbf{y}\|_\infty < 2C\eta, \qquad \max_{i,j\in V}|y_i - y_j| < 4C\eta,
\]
for some constant $C>2$. In particular, being $\eta = o(1)$, we can suppose that $2C\eta\leq 1$. Therefore, we can use \cref{lem: bounds on norms} to derive a bound for $\mathbf{z}$. More precisely, we have that
\[
    \|\mathbf{z}\|_\infty=
    \|\mathcal{Q}_\lambda(P^m \mathbf{y})\|_\infty \le C \max_{i,j\in V}|y_i - y_j|\,\|\mathbf{y}\|_\infty.
\]
Therefore, using the estimates for $\|\mathbf{y}\|_\infty$ and $\max_{i,j\in V}|y_i - y_j|$, we have 
\begin{equation}\label{eq: z uniform}
    \|\mathbf{z}\|_\infty < 8C^3 \eta^2 \leq \eta,
\end{equation}
where the last inequality follows from the fact that $\eta = o(1)$. We then also have that 
\begin{equation}\label{eq11: max z}
    \max_{i,j\in V}|z_i - z_j| < 2\eta
\end{equation}
Thus, proceeding similarly to above and applying \cref{lem: bound on the sum of norms} three times we get that
\begin{equation}\label{eq: est R3}
    \|\mathcal{R}_{3}^{(t)} (\bstate{1})\|_\infty 
    \leq 2C (4C\eta M_\eta)^3 \|\bstate{1}\|_\infty.   
\end{equation}
Note that in order to get the estimate above it is crucial to use that $\eta = o(1)$, which allows us to have the two estimates in \cref{eq: z uniform,eq11: max z}. If this did not hold, we would have a constant in front of the factor $(4C\eta M_\eta)^3$ in \cref{eq: est R3} different than the corresponding one in \cref{eq: est R2}.

\medskip
Note that the case $k>3$ is syntactically equivalent to the case $k=3$. 
In fact, the new vector $\|\mathbf{w}\| = \mathcal{Q}_\lambda(P^{\ell} \mathbf{z})$, for every fixed $\ell$, is such that $\|\mathbf{w}\|_\infty \leq \eta$ and $\max_{i,j\in V}|w_i - w_j| < 2\eta$, exactly as for $\mathbf{z}$ in \cref{eq: z uniform,eq11: max z}.
We underline once more that it is crucial to use that $\eta = o(1)$ to get the equivalent bound of \cref{eq: z uniform} for $\mathbf{w}$.
In conclusion, for every $k$, we iteratively apply \cref{lem: bound on the sum of norms} $k$ times (as done above for $k=1,2,3$) and get that
\[
    \|\mathcal{R}_{k}^{(t)} (\bstate{1})\|_\infty \leq 2C (4C\eta M_\eta)^k \|\bstate{1}\|_\infty,
\]
where we recall that $C>2$ is positive and finite.

\medskip
We now use the estimate above to conclude the proof. In particular, being $\eta M_\eta = o(1)$ by assumption, we have that $|4C\eta M_\eta| <1$ and being $|\lambda|< 1/2$, we can conclude that
\begin{align*}
    \|\mathcal{R}^{(t)} (\bstate{1})\|_\infty
    &\leq \sum_{k=1}^{t} |\lambda|^k \|\mathcal{R}_{k}^{(t)} (\bstate{1})\|_\infty 
    \leq 2C \sum_{k=1}^t |4C\eta M_\eta \lambda|^k \|\bstate{1}\|_\infty 
    \\
    &\leq 2C \sum_{k=1}^\infty |4C\eta M_\eta\lambda|^k \|\bstate{1}\|_\infty 
    \leq 2C\left(\frac{1}{1- 4C|\lambda|\eta M_\eta} - 1\right)\|\bstate{1}\|_\infty 
    \\
    &= 2C\frac{4C|\lambda|\eta M_\eta}{1-4C|\lambda| \eta M_\eta}\|\bstate{1}\|_\infty.
\end{align*}
Since, by assumption, $|\state{1}_i - \rho| < \eta$ for every $i$, then it holds that $\|\bstate{1}\|_\infty < |\rho| + \eta$.
Hence, we get the desired result, i.e, 
\[
 \|\mathcal{R}^{(t)} (\bstate{1})\|_\infty < \frac{8C^2|\lambda| \eta M_\eta}{1- 4C|\lambda| \eta M_\eta}(|\rho| + \eta).
\]
\end{proof}

\section{On the randomness of the initial state}\label{sec: rademacher initial state}
In this section, we prove some concentration and anti-concentration results that are crucial to proving our main theorem (\cref{theorem: main}). 
In \cref{sec: known results} we recall some well-known theorems that we use in the subsequent section to prove our results.
In \cref{sec: concentration} we prove concentration results for the average $\bar{\mu}$ and for the shift $\sigma_{\lambda}(\bstate{0})$, respectively defined in \cref{eq: def mu,eq: sigmalambda}, which are needed to characterize the convergence of our dynamics. 
Moreover, we provide a concentration for $\bstate{1}$, which is needed to apply the results of \cref{sec: Rt} to $\mathcal{R}_\lambda^{(t)}(\bstate{1})$ and show that it is negligible. 
Last, in \cref{sec: anti-concentration}, we provide an anti-concentration result for $|\bar{\mu}|$, which is needed in the proof of \cref{theorem: main} in the case $p=1/2$.

\subsection{Useful concentration inequalities}\label{sec: known results}
In this section, we collect some well-known results that we are going to use in the following. In the proof of \cref{lemma: concentration sigma}, will use a Chernoff's bound for the concentration of binomial random variables, which we recall below.
\begin{theorem}[Chernoff's bound~\cite{10.2307/2236576}]\label{thm: chernoff bound}
Let $X_1,\ldots,X_n$ be independent Bernoulli random variables, i.e., taking values in $\{0,1\}$. Let $X=\sum_{i=1}^n X_i$ be a binomial random variable and let $\mu = \Ex[X]$. For every $\gamma \in [0,1]$ it holds that
\[
    \Pr\left( |X - \mu| \ge \gamma \mu \right) \le 2e^{-\frac{\gamma^2 \mu }{3}}.
\]
\end{theorem}
\noindent
The following theorem instead, Hoeffding's inequality, is a concentration result for the sum of general independent random variables and is used in the proof of \cref{lemma:bound on mu}.
\begin{theorem}[Hoeffding's inequality~\cite{10.2307/2282952}]\label{thm:hoeffding} 
Let $X_{1}, \ldots, X_{n}$ be independent random variables such that $X_{i} \in [L_i, U_i]$. 
Let $X = \sum_{i=1}^n X_i$ and let $\mu = \Ex[X]$. 
Then, for all $\gamma>0$,
\[
    \Pr(| X - \mu | \geq \gamma) 
    \leq 2\exp\left(-\frac{2 \gamma^{2}}{\sum_{i=1}^{n} (U_i - L_i)^2}\right).
\]
\end{theorem}
In \cref{sec: anti-concentration} we will use the Berry-Esseen theorem, a quantitative version of the central limit theorem that, under certain conditions, allows us to approximate a distribution with the standard normal distribution.
Before stating it rigorously, recall that the \textit{cumulative distribution function} $F_X$ of a random variable $X$ evaluated at $x$ is given by $F_X(x) := \Pr(X \le x)$.
For brevity, in what follows, we write cumulative distribution function as CDF.
\begin{theorem}[Berry–Esseen theorem~\cite{esseen1942liapunov,shevtsova2010improvement}]\label{thm:Berry-Esseen}
Let $X_1,\ldots,X_n$ be independent random variables such that $\Ex[X_i] = 0$, $\Ex[X_i^2] = \sigma_i^2 > 0$, and $\Ex[|X_i|^3] = \rho_i < \infty$, for every $i\in\{1,\ldots,n\}$. Also, let
\[
    S_{n} := \frac{\sum_{i=1}^{n}X_{i}}{\sqrt {\sum_{i=1}^{n}\sigma_{i}^{2}}}.
\]
Let us denote by $F_n$ the CDF of $S_n$, and by $\Phi$ the CDF of the standard normal distribution. 
There exists an absolute constant $C_0 \in (0.4097,0.56)$ such that, for every $n$,
\[
    \sup_{x\in \mathbb {R} }\left|F_{n}(x)-\Phi (x)\right|\leq C_0 \psi_0,
\qquad
    \psi_0 := \left( \sum_{i=1}^{n} \sigma_{i}^{2} \right)^{-3/2}
        \cdot \sum_{i=1}^{n} \rho_{i}.
\]
\end{theorem}
\subsection{Concentration results}\label{sec: concentration}
In this section we use \cref{thm: chernoff bound} and \cref{thm:hoeffding} to prove the concentration of $\sigma_\lambda(\bstate{0})$ (see \cref{eq: sigmalambda}), of $\bstate{1}$, and of $\bar{\mu}$ around their expected values.

\subsubsection{Concentration of the shift \texorpdfstring{$\sigma_{\lambda}(\bstate{0})$}{}}
We start by recalling that (see \cref{eq: sigmalambda})
\[
    \sigma_\lambda(\bstate{0}) = \frac{\sum_{i \in V} \left( \sum_{j \in V} W_{ij}\,\sigma_{\lambda;j}(\bstate{0})\right) \state{0}_i }{\sum_{i \in V} D_{ii}\,\state{0}_i},
\]
with
\[
    \sigma_{\lambda;i}(\mathbf{x}^{(0)}) := \frac{c_i(1 - {s(2\lambda)})}{|N_i| - c_i (1- {s(2\lambda)})},  
    \quad\text{and } 
    c_i := |\{\{j,k\} \in N_i\, : \; x_j^{(0)}\neq x_k^{(0)}\}.
\] 
Note that whenever the initial state is random, $c_i$ is a binomial random variable. Hence, the concentration result we want to prove for $\sigma_{\lambda}(\bstate{0})$ comes as a consequence of \cref{thm: chernoff bound}.
\begin{proposition}[Concentration of $\sigma_\lambda(\bstate{0})$]\label{lemma: concentration sigma}
Let $\lambda\in \left(-\frac{1}{2},\frac{1}{2}\right)$. Let $\bstate{0}$ be $p$-Rademacher random vector according to \cref{def: p rademacher vector}  for $p\in [0,1]$ and let $\sigma_{\lambda;i}(\bstate{0})$, for every $i\in V$, as in \cref{eq: def mui sigmai}. 
If $|\nbr{i}| = \omega(\ln n)$, for every $i \in V$, it holds that
\begin{equation}\label{eq: conc sigmai}
    \Pr\left(
        \forall i\in V,\,\, \sigma_{\lambda;i}(\bstate{0}) = \frac{2p(1-p)(1-{s(2\lambda)})}{1-2p(1-p)(1-{s(2\lambda)})} \cdot (1\pm \gamma)
    \right) = 1-n^{-\omega(1)},
\end{equation}
for any  $\gamma\in[0,1]$, $\gamma= o(1)$ and  such that
$\gamma = \omega(\sqrt{\log (n)/d_{\min}})$, where  $d_{\min} := \min_{i\in V} |\nbr{i}|$.

As a consequence, with the same choice of $\gamma$, we have
\begin{equation}\label{eq: conc sigma}
    \Pr\left(
        \sigma_{\lambda}(\bstate{0}) = \frac{2p(1-p)(1-{s(2\lambda)})}{1-2p(1-p)(1-{s(2\lambda)})} \cdot (1\pm \gamma)
    \right) = 1 - n^{-\omega(1)},
\end{equation}
where $\sigma_\lambda(\bstate{0})$ is as defined in \cref{eq: sigmalambda}. 
\end{proposition}
\begin{proof}
We start by proving the bound in \cref{eq: conc sigmai}. Recall that for $\lambda\in \left(-{1}/{2},{1}/{2}\right)$ and for every $i\in V$,
\(
    c_i := \big| \{ \{j,k\} \in \nbr{i} : \state{0}_j \neq \state{0}_k \} \big|.
\)
Denoting by $\mathbf{\chi}\{\mathcal{F}\}$ the indicator function of the event $\mathcal{F}$, we can rewrite $c_i$ as
\[
    c_i = \sum_{\{j,k\}\in\nbr{i}} \mathbf{\chi}\{\state{0}_j \neq \state{0}_k\}.
\]
In particular, $c_i$ is a binomial random variable with a number of trials $|\nbr{i}|$ and a probability of success $2p(1-p)$. 
Therefore, we have that
\[  
    \Ex[c_i] = 2p(1-p)|\nbr{i}|.
\]
From \cref{thm: chernoff bound}, for every $\gamma \in [0,1]$, we get that
\[
    \Pr\left( \left| c_i - \Ex[c_i] \right| > \gamma\Ex[c_i] 
    \right) 
    \leq 2e^{-\frac{\gamma^2}{3} \Ex[c_i]} .
\]
Using now that $d_{\min} := \min_{i \in V} |\nbr{i}| = \omega(\ln n)$ by assumption, it follows that $\Ex[c_i] = \omega(\ln n)$ since $p$ is constant. 
By choosing $\gamma\in[0,1]$ such that
$\gamma = \omega\big(\sqrt{\log(n)/d_{\min}}\big)$ and $\gamma = o(1)$,  we find that 
\[
    \Pr\left( \left| c_i - \Ex[c_i] \right| > \gamma\Ex[c_i] 
    \right) 
    \leq n^{-\omega(1)}.
\]
By taking a union bound over all \nodes $i\in V$, it follows that
\[
    \Pr\left( \exists i\in V : \left| c_i - \Ex[c_i] \right| > \gamma\Ex[c_i] 
    \right) 
    \leq \sum_{i\in V} \Pr\left( \left| c_i - \Ex[c_i] \right| > \gamma\Ex[c_i] 
    \right)
    = n^{-\omega(1)}.
\]
Being $\Ex[c_i] = 2p(1-p)|\nbr{i}|$, we deduce that $c_i = 2p(1-p)|\nbr{i}|(1\pm \gamma)$ with high probability.
By using the previous probabilistic bounds on $c_i$ and the definition of $\sigma_{\lambda;i}(\bstate{0})$, we conclude that
\begin{equation}\label{eq: concentration sigmai}
    \Pr\left(
        \forall i\in V,\,\, \sigma_{\lambda;i}(\bstate{0}) = \frac{2p(1-p)(1-{s(2\lambda)})}{1-2p(1-p)(1-{s(2\lambda)})} \cdot (1\pm \gamma)
    \right) = 1-n^{-\omega(1)}.
\end{equation}

Recall the definition of $\sigma_\lambda(\bstate{0})$ given in \cref{eq: sigmalambda}, namely
\[
    \sigma_\lambda(\bstate{0}) := \frac{\sum_{i \in V} \left( \sum_{j \in V} W_{ij}\,\sigma_{\lambda;j}(\bstate{0})\right) \state{0}_i }{\sum_{i \in V} D_{ii}\,\state{0}_i}.
\]
Note that if $\sigma_{\lambda;i}(\bstate{0})$ assumes the same value for every $i \in V$, then $\sigma_\lambda(\bstate{0})$ equals such a value. 
From \cref{eq: concentration sigmai} we know that the value of $\sigma_{\lambda;i}(\bstate{0})$ is concentrated around $\frac{2p(1-p)(1-{s(2\lambda)})}{1-2p(1-p)(1-{s(2\lambda)})}$, for every $i\in V$, with high probability. 
Then, the same holds for $\sigma_\lambda(\bstate{0})$, i.e., 
\[
    \Pr\left(
        \sigma_{\lambda}(\bstate{0}) = \frac{2p(1-p)(1-{s(2\lambda)})}{1-2p(1-p)(1-{s(2\lambda)})} \cdot (1\pm \gamma)
    \right) = 1 - n^{\omega(1)}.
\]
\end{proof}

\subsubsection{Concentration of the average \texorpdfstring{$\bar\mu$}{}}
We now discuss a concentration result of the average $\bar{\mu}$, defined in \cref{eq: def mu}, which is a consequence of \cref{thm:hoeffding}.

\begin{proposition}[Concentration of $\bar\mu$]\label{lemma:bound on mu}
Let $\bstate{0}$ be a $p$-Rademacher random vector. 
For every fixed $i \in V$, let $\bar\mu_i$ be as defined in \cref{eq: def mui sigmai} and let $\mu = \Ex[\bar\mu_i]=2p-1$.
Then, for $\gamma\geq 0$, we have
\begin{equation}\label{eq: conc mui}
\Pr\left[
    \left| \bar\mu_i - \mu \right|<2\gamma
    \right]
    \geq
    1-2 \exp\left(-2 \gamma^2 \Big(\sum_{j\in V} \left({W_{ij}}/{D_{ii}}\right)^2\Big)^{-1}\right).
\end{equation}
Moreover, let $\bar{\mu}$ be as defined in \cref{eq: def mu}.
Then, for $\gamma\geq 0$, we have
\begin{equation}\label{eq: conc mu}
\Pr\left[
    \left| \bar\mu - \mu \right|<2\gamma
    \right]
    \geq
    1-2\exp\left(-\frac{\gamma^2}{2} \left(\sum_{i\in V} D_{ii}^2 \,/\, \Big( \sum_{i\in V} D_{ii}\Big)^2 \right)^{-1}
    \right).
\end{equation}
\end{proposition}
\begin{proof}
We start by proving \cref{eq: conc mui}. For every \nodes $i,j\in V$ let us define
\begin{equation}\label{eq:defns of aj Zj Sn}
    a_{ij}:=\frac{W_{ij}}{D_{ii}},
    \quad
    Z_{j} := \frac{\state{0}_j + 1}{2},
    \quad\text{and }
    X_i := \sum_{j\in V} a_{ij}Z_{j}.
\end{equation}
Note that, since $\bstate{0}$ is a $p$-Rademacher random vector, the random variables $(Z_{j})_{j\in V}$ are i.i.d.\ Bernoulli random variable of parameter $p$. 
Moreover, note that $\sum_{j \in V} a_{ij} = 1$ (see \cref{eq:graph D}). 
Hence $\Ex[X_i] = p$.
Therefore, by defining $(X_i)_j = a_{j} Z_{ij}$ and noting that $(X_i)_j\in [0, a_{j}]$, we can apply \cref{thm:hoeffding} to $X_i = \sum_{j\in V} (X_i)_j$ to deduce that for all $\gamma \geq 0$, it holds
\begin{equation}\label{eq:bound of Sn-p}
    \Pr\left[
        |X_i - p| \ge \gamma 
    \right]
    \leq 2 \exp\left(-\frac{2 \gamma^2}{\sum_{j\in V} a_{ij}^2}\right).
\end{equation}
By putting the definitions in \cref{eq:defns of aj Zj Sn} into \cref{eq:bound of Sn-p}, we have
\[
    \Pr\Bigg[
        \Bigg|\sum_{j\in V} a_{ij} \left(\frac{\state{0}_j + 1}{2}\right) - p \Bigg| \geq \gamma 
    \Bigg]
    = \Pr\Bigg[
        | \bar\mu_i - \mu |
        \ge 2\gamma
    \Bigg]
    \leq 2 \exp\left(-\frac{2 \gamma^2}{\sum_{j\in V} a_{ij}^2}\right)
\]
which is equivalent to
\[
\Pr\left[
    \left| 
        \bar\mu_i - \mu \right|
        <2 \gamma
    \right]
    \geq
    1-2\exp\left(- \frac{2 \gamma^2}{\sum_{j\in V} (W_{ij}/D_{ii})^2}
    \right).
\]

We now discuss the proof of \cref{eq: conc mu}. For every $i\in V$, we define
\[
    a_i = \frac{D_{ii}}{\sum_{j \in V} D_{jj}},
    \quad
    X_i = \frac{\state{0}_i + 1}{2},
    \quad\text{and }
    X := \sum_{i \in V} a_{i}Z_{i}.
\]
Then, as done for \cref{eq: conc mui}, an application of \cref{thm:hoeffding} gives that 
\[
    \Pr\left[
    \left| \bar\mu - \mu \right|<2\gamma
    \right]
    \geq
    1-2\exp\left(-\frac{\gamma^2}{2} \left(\sum_{i\in V} D_{ii}^2 \,/\, \Big( \sum_{i\in V} D_{ii}\Big)^2 \right)^{-1}
    \right).
\]
\end{proof}

\subsubsection{Concentration of the state \texorpdfstring{$\bstate{1}$}{}}
We now combine \cref{lemma: concentration sigma,lemma:bound on mu} to study the concentration of $\bstate{1}$.

\begin{proposition}[Concentration of $\bstate{1}$]\label{prop: bounding x1}
Let $\lambda\in \left(-\frac{1}{2},\frac{1}{2}\right)$ and $\bstate{0}$ be a $p$-Rademacher random vector for $p\in [0,1]$.
Let $\mu := 2p-1$ and let $\sigma_\lambda(\bstate{0})$ be as defined in \cref{eq: sigmalambda}. 
If $|\nbr{i}|=\omega(\log n)$, for every $i \in V$, then it holds that
\[
    \Pr\left( 
        \forall i \in V,\,
        \left| \state{1}_i - \mu ( 1 + \sigma_{\lambda}(\bstate{0})) \right| \le 6 \gamma 
    \right)
    \ge 1-2  n \exp\left(-2 \gamma^2 \Big(\max_{i\in V} \sum_{j\in V} \left({W_{ij}}/{D_{ii}}\right)^2\Big)^{-1}\right),
\]
for any  $\gamma\in[0,1]$, $\gamma= o(1)$ and  such that
$\gamma = \omega(\sqrt{\log (n)/d_{\min}})$, where  $d_{\min} := \min_{i\in V} |\nbr{i}|$.
\end{proposition}
\begin{proof}
From \cref{lemma: value of xi(1)}, we know that for all $i \in V$
\[
    \state{1}_{i} = \bar\mu_i \cdot \left( 1 + \sigma_{\lambda;i}(\bstate{0}) \right),
    \quad\text{with }
    \bar\mu_i := \frac{1}{D_{ii}}\sum_{j\in V} W_{ij}\,\state{0}_j,
\]
see \cref{eq: def mui sigmai} for the definition of $\sigma_{\lambda;i}(\bstate{0})$.
Via \cref{lemma: concentration sigma}, for $\gamma = \omega(\sqrt{\log (n)/d_{\min}})$ and $\gamma=o(1)$, we get that
\[
    \sigma_{\lambda;i}(\bstate{0}) \!=\! \frac{2p(1-p)(1-{s(2\lambda)})}{1-2p(1-p)(1-{s(2\lambda)})} \cdot (1\pm \gamma)
    \,\text{ and }\,
    \sigma_{\lambda}(\bstate{0}) \!=\! \frac{2p(1-p)(1-{s(2\lambda)})}{1-2p(1-p)(1-{s(2\lambda)})} \cdot (1\pm \gamma)
\]
for all $i$ simultaneously, with high probability.
Therefore, with high probability, it follows that
\[
    | \sigma_{\lambda;i}(\bstate{0}) - \sigma_{\lambda}(\bstate{0}) | \leq 2\gamma.
\]
Thus, for every fixed $i$ we have
\begin{align*}
    \left| \state{1}_i  - \mu ( 1 + \sigma_{\lambda}(\bstate{0})) \right|
    &=\left|\bar\mu_i ( 1 + \sigma_{\lambda;i}(\bstate{0})) - \mu ( 1 + \sigma_{\lambda}(\bstate{0})) \right|
    \\
    &=\left|\bar\mu_i ( 1 + \sigma_{\lambda}(\bstate{0})) - \mu ( 1 + \sigma_{\lambda}(\bstate{0})) + \bar\mu_i (\sigma_{\lambda;i}(\bstate{0})-\sigma_{\lambda}(\bstate{0})) \right|
    \\
    &\leq \left|( 1 + \sigma_{\lambda}(\bstate{0})) (\bar\mu_i-\mu) \right| + 2\gamma \left| \bar\mu_i \gamma \right|
    \leq 2|\bar\mu_i-\mu| + 2\gamma,
\end{align*}
since $|\bar\mu_i|\le 1$ and $0< 1 + \sigma_{\lambda}(\bstate{0})< 2$.
By applying \cref{lemma:bound on mu} (\cref{eq: conc mui}), we get that $|\bar\mu_i-\mu| < 2\gamma$ with high probability, and in particular we obtain
\[
    \Pr\left(
        \left| \state{1}_i - \mu ( 1 + \sigma_{\lambda}(\bstate{0})) \right| \le 6 \gamma
    \right)
    \ge 1-2\exp\left(-2 \gamma^2 \Big(\sum_{j\in V} \left({W_{ij}}/{D_{ii}}\right)^2\Big)^{-1}\right).
\]
We conclude the proof by applying a union bound over all the \nodes in $V$, which implies that
\begin{align*}
    \Pr\left( \forall i\in V, \,
    \left| \state{1}_i - \mu ( 1 + \sigma_{\lambda}(\bstate{0})) \right| \le 6 \gamma \right) 
    &\ge 1-2  \sum_{i\in V}\exp\left(-2 \gamma^2 \Big(\sum_{j\in V} \left({W_{ij}}/{D_{ii}}\right)^2\Big)^{-1}\right)
    \\
    &\ge 1-2  n \exp\left(-2 \gamma^2 \Big( \max_{i\in V} \sum_{j\in V} \left({W_{ij}}/{D_{ii}}\right)^2\Big)^{-1}\right).
\end{align*}
\end{proof}

\subsection{Anti-concentration results}\label{sec: anti-concentration}

In this subsection, we use \cref{thm:Berry-Esseen} to prove an anti-concentration result for the average $|\bar{\mu}|$, which is needed in the proof of \cref{theorem: main} in the case $p=1/2$.
Before doing that, we prove a Berry-Esseen type of result for folded distributions. 
We use the notation CDF for the cumulative distribution function. 

\begin{lemma}[Berry--Esseen for folded distributions]\label{prop:folded Berry-Essen}
Let $X_1,\ldots,X_n$ and $S_n$ be assumed as in \cref{thm:Berry-Esseen}. Moreover, let us denote by $G_n$ the CDF of $|S_n|$ and by $\Psi$ the CDF of the folded standard normal distribution. 
Then, for every $n$,
\[
    \sup_{x\in \mathbb {R} }\left|G_{n}(x)-\Psi(x)\right|\leq 2C_0 \psi_0,
\]
with $C_0,\psi_0$ as in \cref{thm:Berry-Esseen}.
\end{lemma}
\begin{proof}
We are going to use \cref{thm:Berry-Esseen}.
Since both the normal distribution and $S_n$ have mean zero and are symmetric, one can see that for $x\geq 0$
\[
    G_n(x) = 2F_n(x) - 1
    \quad\text{and}\quad
    \Psi(x) = 2\Phi(x)-1,
\]
where $F_n$ and $\Phi$ are the  CDFs of $S_n$ and of the standard normal distribution, respectively.
Then, by \cref{thm:Berry-Esseen}, we obtain
\begin{align*}
    \sup_{x\in \mathbb {R} }\left|G_{n}(x)-\Psi(x)\right|
    &= \sup_{x\in \mathbb {R} }\left|\big(2F_n(x) - 1\big)-\big(2\Phi(x)-1\big)\right|\\
    &= 2 \sup_{x\in \mathbb {R} } \left|F_n(x)-\Phi(x) \right|
    \leq 2 C_0\psi_0.
\end{align*}
\end{proof}

\begin{proposition}[Anti-concentration of $|\bar\mu|$]\label{thm:lower bound mu}
Let $\bstate{0}$ be a Rademacher random vector. Let $\bar\mu$ as defined in \cref{eq: def mu}, i.e., 
\[
    \bar{\mu} = \frac{\sum_{i\in V} D_{ii}x_i^{(0)}}{\sum_{i\in V}D_{ii}},
\]
where $D=(D_{ii})$ is the transition matrix defined in  \cref{eq:graph D}
For any $a>0$, the following holds:
\[
    \Pr\left( |\bar\mu| > a  
    \right) 
    \geq 1 
    -\Psi\left( a\cdot \frac{\sum_{i\in V} D_{ii}}{\sqrt{\sum_{i\in V} D_{ii}^2}} \right)
    -2C_0 \left( \sum_{i \in V} D_{ii}^2 \right)^{-3/2}
        \cdot \sum_{i \in V} D_{ii}^3,
\]
where $\Psi$ is the CDF of the folded standard normal distribution and $C_0$ is the constant from \cref{thm:Berry-Esseen}.
\end{proposition}

\begin{proof}
By the definition of $\bar\mu$ given in \cref{eq: def mu}, we get
\begin{align*}
    \Pr\left( |\bar\mu| > a 
    \right)
    &= \Pr\left( \left| \frac{\sum_{i\in V} D_{ii} \state{0}_i}{\sum_{i\in V} D_{ii}} \right| > a 
    \right)
    = \Pr\left( \left| {\sum_{i\in V}D_{ii}\,\state{0}_i} \right| > a \sum_{i\in V}D_{ii} 
    \right)
    \\
    &= \Pr\left( \frac{\left| \sum_{i\in V}D_{ii}\,\state{0}_i \right|}{\sqrt{\sum_{i\in V}D_{ii}^2}} > a \cdot \frac{\sum_{i\in V} D_{ii}}{\sqrt{\sum_{i \in V} D_{ii}^2}} 
    \right)
    = \Pr\left( |S_n| > \alpha \right),
\end{align*}
where
\[
    S_{n} := \frac{\sum_{i\in V} D_{ii}\,\state{0}_i}{\sqrt{\sum_{i\in V}D_{ii}^2}}
    \quad\text{and}\quad
    \alpha := a\cdot \frac{\sum_{i \in V} D_{ii}}{\sqrt{\sum_{i\in V} D_{ii}^2}}.
\]

In the remainder, we show how to apply \cref{prop:folded Berry-Essen} to the previous equation.
For every $i\in V$, let us the define the random variable $X_i := D_{ii} \state{0}_i$. 
Note that we have:
\begin{itemize}
\item[(i)] $\Ex[X_i]=0$, since $\bstate{0}$ is a Rademacher random vector;
\item[(ii)] $\Ex[X_i^2] = D_{ii}^2 > 0$, since the hypergraph is connected, i.e., $D_{ii}>0$ for every $i$;
\item[(iii)] $\Ex[|X_i|^3] = D_{ii}^3 < \infty$.
\end{itemize}
Moreover, note that the CDF of $|S_n|$ is given by
\[
    G_n(\alpha) = \Pr\left(|S_n|\leq \alpha \right) 
    = \Pr\left( |\bar\mu| \leq a
    \right) 
    = 1 - \Pr\left( |\bar\mu| > a
    \right).
\]
Then, by \cref{prop:folded Berry-Essen}, we have
\begin{align*}
    \Pr\left( |\bar\mu| > a
    \right) 
    &= 1 - G_n(\alpha) 
    = 1 - \Psi(\alpha) + \Psi(\alpha) - G_n(x) 
    \\
    &\geq 1 - \Psi(\alpha) - 2C_0 \left( \sum_{i \in V} \Ex[X_i^2] \right)^{-3/2}
        \cdot \sum_{i \in V} \Ex[|X_i|^3].
\end{align*}
By using the values of $\Ex[X_i^2]$ and $\Ex[|X_i|^3]$, and the definition of $\alpha$ given above, we conclude that
\[
    \Pr\left( |\bar\mu| > a
    \right)
    = 1 - \Psi\left( 
        a\cdot \frac{\sum_{i\in V} D_{ii}}{\sqrt{\sum_{i \in V}D_{ii}^2}} 
    \right) 
    -2C_0 \left( \sum_{i \in V} D_{ii}^2 \right)^{-3/2}
        \cdot \sum_{i \in V} D_{ii}^3.
\]
\end{proof}
\section{On a class of hypergraphs that satisfy the assumptions}\label{sec: erdos renyi hypergraph}

In this section, we show the existence of a family of hypergraphs that satisfy \cref{assumptions on Gamma} and \cref{extra assumptions on Gamma}.
Such a family naturally generalizes the Erd\H{o}s-R\'enyi random graph model to hypergraphs, by considering that each triple of \nodes forms a hyperedge with probability $p$, independently of the others.

\begin{definition}[Erd\H{o}s-R\'enyi hypergraph]\label{def: er hypergraphs}
A 3-uniform Erd\H{o}s-R\'enyi hypergraph $\Gamma=(V,H)$ of parameters $n$ and $p$ is a simple hypergraph with $|V|=n$ and such that $\{i,j,k\} \in H$ with probability $p$, independently for each triple of \nodes $i,j,k\in V$.
\end{definition}

\begin{theorem}[Erd\H{o}s-R\'enyi hypergraphs satisfy assumptions]\label{prop: erh}
Let $\Gamma$ be a 3-uniform Erd\H{o}s-R\'enyi hypergraph of parameters $n$ and $p$ (\cref{def: er hypergraphs}).
For every $\gamma\in (0,1]$, there exists a constant $K>0$ such that, with probability at least $1-O(n^{-\gamma})$, the hypergraph $\Gamma$ satisfies \cref{assumptions on Gamma} and \cref{extra assumptions on Gamma} for every 
\begin{equation}\label{eq: condition p}
    p\ge K\frac{\sqrt[3]{\log n}}{n^{1-(2\gamma/3)}}.
\end{equation}
\end{theorem}

We devote the remainder of this section to the proof of \cref{prop: erh}. Before proceeding with the proof, though, we state a classical result about the perturbation theory of eigenvalues of Hermitian matrices, that will be used in our argument.
Note that the matrices we deal with in this paper are real, hence they are Hermitian whenever they are symmetric.

\begin{theorem}[Eigenvalue Perturbation {\cite[Corollary~4.10]{stewart1990matrix}}]\label{thm: symmetric perturbation eigenvalues}
For a matrix $M$, let $\lambda_i(M)$ denote the $i$-th largest eigenvalue of $M$. Let $A,E \in \mathbb{C}^{n\times n}$ be Hermitian matrices. Then 
\[
    \max_{i=1,\ldots,n} |\lambda_i(A+E)-\lambda_i(A)| \le \|E\|_2.
\]
\end{theorem}

We are also going to use an additional result, that compares the spectra of different matrix representations of the same graph, namely the adjacency and transition matrices.
The proof uses an argument seen in~\cite{lutzeyer2017comparing}, but refined since the transition matrix $P$ is not symmetric in our case.

In order to prove that $\Gamma$ satisfies the \cref{assumptions on Gamma} and \cref{extra assumptions on Gamma}, we need information on $\nu$, namely on the second largest eigenvalue (in absolute value) of the transition matrix $P$ of a random walk on the motif graph $G(\Gamma)$.
In particular, exploiting \cref{thm: symmetric perturbation eigenvalues}, in the next lemma we are going to first give an estimate of the eigenvalues of the weighted adjacency matrix $W$ by considering the eigenvalues of its expected matrix and by providing error bounds with high probability.
Then, using \cref{thm: adjacency vs transition matrix spectra}, we will approximate the eigenvalues of $P$ with a mapping of the corresponding eigenvalues of $W$, which comes at the cost of an additional error.
Note that the next proposition also derives from \cite[Theorem~1]{chung2011spectra}. In fact, our proposition restricts to motif graphs of Erd\H{o}s-R\'enyi hypergraphs. However, we keep it since our proof uses slightly different tools and is much more compact.

\begin{lemma}[Spectra of Erd\H{o}s-R\'enyi motif graphs]\label{prop: erh nu}
Let $\Gamma$ be a Erd\H{o}s-R\'enyi hypergraph with parameter $p$.
Let $G=G(\Gamma)$ be the motif graph of $\Gamma$. 
Let $\lambda_1(W) \ge \lambda_2(W) \ge \ldots \ge \lambda_n(W)$ be the eigenvalues of the weighted adjacency matrix $W$ of $G$, counted with multiplicities. 
For every $i>1$ it holds that, with high probability,
\[
    \Big|\, \lambda_1(W) - p(n-2)(n-1) \Big| \leq \sqrt{{n \log(n)/p}}, 
    \qquad
    \Big|\, \lambda_i(W) - \left( -p(n-2) \right) \Big| \leq \sqrt{{n \log (n)/p}}.
\]
\end{lemma}
\begin{proof}
Let us consider the weighted adjacency matrix $W$ of the motif graph $G=G(\Gamma)$ of an Erd\H{o}s-R\'enyi hypergraph $\Gamma$ (\cref{def: er hypergraphs}).
Note that, for every fixed pair of \nodes $i,j \in V$, the corresponding entry of the matrix $W$ is $W_{ij} = \sum_{k\in V}A_{i,j,k}$.
Each entry $W_{ij}$ is a binomial random variable and $\Ex[W_{ij}] = p(n-2)$.
By an application of \cref{thm: chernoff bound}, with $\gamma = \sqrt{\log(n)/ pn (p(n-2))^2}$, and of the union bound over all pairs of \nodes, it follows that
\begin{equation}\label{eq: chernoff wij}
    \Pr\left(
        \exists\, i,j \in V \,:\,
        \left| W_{ij} - p(n-2) \right| \ge \sqrt{{\log (n)/p n}}
    \right) \le n^{-\omega(1)}.
\end{equation}
Let us now write $W$ as the sum of the expected adjacency matrix and of a stochastic matrix that adjusts the actual edge weights, namely, with high probability,
\[
    W=\widehat W + \widetilde W, 
    \qquad 
    \widehat W_{ij} = p(n-2),
    \quad\text{and } |\widetilde W_{ij}| < 
    \sqrt{{\log(n)/p n}},
\]
for every $i\neq j$ and $\widehat W_{ii}=\widetilde W_{ii}=0$ on the diagonal.
In a more compact form, we can write $\widehat W = p(n-2)[J - I]$, where $J$ is the matrix of all 1s and $I$ is the identity matrix. Therefore, by denoting with $\lambda_i(W)$ the $i$-th largest eigenvalue of $W$, the matrix $\widehat W$ has eigenvalues%
\footnote{The rank of $J$ is $1$, i.e., it has one nonzero eigenvalue equal to $n$ with eigenvector $\mathbf{1}$, and all the remaining eigenvalues are $0$. Subtracting the identity matrix $I$ from $J$ results in a shift of all eigenvalues $\lambda$ by $-1$, because $Ax=(J-I)x=Jx-x=(\lambda-1)x$. 
Therefore the eigenvalues of $J-I$ are $n-1$, and $-1$ with multiplicity $n-1$. We conclude multiplying by $p(n-2)$.}
\[
    \lambda_1(\widehat W) = p(n-2)(n-1),
    \qquad\text{and }\quad
    \lambda_2(\widehat W)=\lambda_3(\widehat W)=\ldots=\lambda_n(\widehat W) = -p(n-2).
\]
Moreover, by denoting with $\| \cdot \|_F$ the Frobenius norm, it follows that, with high probability,
\[
    \|\widetilde W\|_2 \leq \|\widetilde W\|_F 
    = \sqrt{\sum_{i=1}^n\sum_{i=j}^n |\widetilde W_{ij}|^2}
    \le \sqrt{\sum_{i=1}^n\sum_{i=j}^n \frac{\log n}{p n}}
    = \sqrt{n^2 \frac{\log n}{p n}}
    = \sqrt{\frac{n \log n}{p}}.
\]
Finally, note that $\widehat W, \widetilde W$ are real and symmetric matrices since $G(\Gamma)$ is an undirected graph.
Therefore, by applying \cref{thm: symmetric perturbation eigenvalues}, for every $i$ we get that 
\[
    \Big|\, \lambda_i(W) - \lambda_i(\widehat W) \Big| \leq \sqrt{{n \log (n)/p}},
\]
which combined with the previous observation on the eigenvalues of $\widehat W$ concludes the proof.
\end{proof}

\begin{lemma}[Spectra of adjacency and transition matrices]\label{thm: adjacency vs transition matrix spectra}
Let $G$ be a graph, and let $d_{\min},d_{\max}$ respectively be the minimum and maximum degrees of \nodes in $G$. 
Let $W$ be the weighted adjacency matrix of $G$, and let $P$ be the transition matrix of a random walk on $G$. 
Denote by $\lambda_i(A)$ the $i$-th largest eigenvalue of matrix $A$. 
For every $i=1,\ldots,n$, it holds that%
\footnote{Note that $d_{\max}/d_{\min} = \Delta$, as defined in \cref{eq: def delta}.}
\[
    \left|\lambda_{i}(P)-\frac{2}{d_{\max}+d_{\min}}\cdot\lambda_{i}(W)\right| 
    \; \le \; \frac{d_{\max}}{d_{\min}} \cdot \frac{d_{\max}-d_{\min}}{d_{\max}+d_{\min}}.
\]
\end{lemma}

\begin{proof}
Recall that $P:=D^{-1}W$, where $D$ is the diagonal degree matrix of $G$. While $W$ is symmetric, in general $P$ is not, unless the graph is regular.
Let us define 
\[
    \widehat{P} := cW,
\] 
a rescaled version of $W$, for some positive $c\in \mathbb{R}$ that will be specified later. In what follows, we will compare $\hat{P}$ with $P$.
Note that $\widehat{P}$ is real and symmetric and that the eigenvalues and eigenvectors of $W$ and $\widehat{P}$ are related as follows:
\[
    Wv_i = \lambda_i(W) v_i \iff
    \widehat{P}v_i = (c \lambda_i(W)) v_i.
\]
Furthermore, let us also define 
\[
    \widetilde{P} := D^{-\frac{1}{2}}WD^{-\frac{1}{2}},
\]
a normalized version of $W$ that will also be used in the comparison between $P$ and $\widehat{P}$.
Observe that $\widetilde{P}$ is symmetric.
Moreover, note that the eigenvalues and eigenvectors of $P$ and $\widetilde{P}$ are related as follows:
\[
    Pv_i = \lambda_i(P) v_i \iff \widetilde{P} \cdot D^{\frac{1}{2}}v_i = \lambda_i(P)\cdot D^{\frac{1}{2}}v_i.
\]
Since $\widehat P$ and $\widetilde P$ are real and symmetric, by using \cref{thm: symmetric perturbation eigenvalues} and the aforementioned relations between the eigenvalues of the matrices, it can be seen that:
\begin{equation}\label{eq: lowerbound of Phat - P}
\|\widehat{P}-\widetilde{P}\|_2 
\ge |\lambda_i(\widetilde{P} + \widehat{P} - \widetilde{P} ) - \lambda_i(\widetilde{P})| =  |\lambda_i(\widehat{P})-\lambda_i(\widetilde{P})|
= |c\,\lambda_i(W)-\lambda_i(P)|.
\end{equation}
Observe that, since $D$ is diagonal, the $(i,j)$-entry of the matrix $\widehat{P}-\widetilde{P}$ can be written as
\[
    [\widehat{P}-\widetilde{P}]_{ij}
    = [cW-D^{-\frac{1}{2}}WD^{-\frac{1}{2}}]_{ij}
    = W_{ij}\left(c - D_{ii}^{-\frac{1}{2}} D_{jj}^{-\frac{1}{2}}\right).
\]
Recall that $W$ is symmetric, hence $\|W\|_{1}=\|W\|_{\infty}=d_{\max}$.
Thus, using Holder's inequality to bound $\|A\|_2 \le \sqrt{\|A\|_1 \cdot \|A\|_\infty}$, we have:
\begin{align*}
\|\widehat{P}-\widetilde{P}\|_{2} 
=\|cW-D^{-\frac{1}{2}}WD^{-\frac{1}{2}}\|_{2}
&\le \sqrt{\|cW-D^{-\frac{1}{2}}WD^{-\frac{1}{2}}\|_{1}\cdot\|cW-D^{-\frac{1}{2}}WD^{-\frac{1}{2}}\|_{\infty}}\\
&\leq\sqrt{\|W\|_{1}\cdot\|W\|_{\infty}} \cdot \max_{i,j=1,\ldots,n}\left\{ |D_{ii}^{-\frac{1}{2}}D_{jj}^{-\frac{1}{2}}-c|\right\}
\\
&= d_{\max}\cdot \max_{i,j=1,\ldots,n}\left\{ D_{ii}^{-\frac{1}{2}}D_{jj}^{-\frac{1}{2}}|1-cD_{ii}^{\frac{1}{2}}D_{jj}^{\frac{1}{2}}|\right\} 
\\
&\leq \frac{d_{\max}}{d_{\min}} \cdot \max\left\{|1-cd_{\max}|,\,|1-cd_{\min}|\right\},
\end{align*}
where we used that $d_{\min}\leq D_{ii}\leq d_{\max}$ for every $i=1,\ldots, n$.
After an optimization over $c$, which yields $c=2/(d_{\max}+d_{\min})$, we find 
\begin{equation}\label{eq: upperbound of Phat - P}
\|\widehat{P}-\widetilde{P}\|_{2} \le 
\frac{d_{\max}}{d_{\min}}\cdot\frac{d_{\max}-d_{\min}}{d_{\max}+d_{\min}}.
\end{equation}
Therefore, by combining \cref{eq: lowerbound of Phat - P,eq: upperbound of Phat - P} with the above choice of $c$, we have the desired bound
\[
    \left|\lambda_i(P)-\frac{2}{d_{\max}+d_{\min}}\cdot\lambda_i(W)\right| 
    \; \leq \; \frac{d_{\max}}{d_{\min}}\cdot \frac{d_{\max}-d_{\min}}{d_{\max}+d_{\min}}\;,
\]
which concludes the proof.
\end{proof}

We are now ready to prove \cref{prop: erh}.

\begin{proof}[Proof of \cref{prop: erh}]
Note that, for every fixed $i\in V$, we have that $|\nbr{i}|$ can be expressed as a binomial random variable, where each pair of \nodes is counted with probability $p$ and $\Ex[|\nbr{i}|] = (n-2)(n-1) p / 2$.
By applying \cref{thm: chernoff bound} and a union bound over all \nodes{}, it holds that, with high probability,
\begin{equation}\label{eq: erh degree}
    \left| \, |\nbr{i}| - (n-2)(n-1) p/2 \,\right| \leq \sqrt{n \log (n)/p},
    \quad \forall i \in V.
\end{equation}
Thus, by calling $d_{\min},d_{\max}$ the minimum and maximum degrees in the motif graph $G(\Gamma)$, and by using that $D_{ii} = 2|\nbr{i}|$ for every $i$, we have that, with high probability:
\begin{gather}\label{eq: dmin up lw bds}
    p(n-2)(n-1) - \sqrt{n \log (n)/p} \leq \; d_{\min} \leq p(n-2)(n-1),
\\
\label{eq: dmax up lw bds}
    p(n-2)(n-1) \leq  \; d_{\max} \leq p(n-2)(n-1) + \sqrt{n \log (n)/p}.
\end{gather}

Let $P$ be the transition matrix of a random walk on the motif graph of $\Gamma$.
Let $\lambda_1(P) \ge \lambda_2(P) \ge \ldots \ge \lambda_n(P)$ be the eigenvalues of $P$, counted with multiplicities, and let $\nu := \max\big(|\lambda_2(P)|,\, |\lambda_n(P)|\big)$.
By applying \cref{prop: erh nu,thm: adjacency vs transition matrix spectra} we get that, with high probability,
\begin{align*}
    \left| \nu-\frac{2p(n-2)}{d_{\max}+d_{\min}}\right|
    &\leq \frac{2}{d_{\max}+d_{\min}} \cdot \sqrt{\frac{n\log n}{p}} + \frac{d_{\max}}{d_{\min}} \cdot \frac{d_{\max}-d_{\min}}{d_{\max}+d_{\min}}
    \\
    &\leq \frac{1}{p(n-2)(n-1)-\sqrt{\frac{n \log n}{p}}}\sqrt{\frac{n\log n}{p}}\left(1+\frac{d_{\max}}{d_{\min}}\right)
    \le C \sqrt{\frac{\log n}{n^3}} p^{-3/2},
\end{align*}
where $C$ is a positive constant independent of $n$ and where we have used that, since $p$ satisfies \cref{eq: condition p}, then $p(n-2)(n-1) -\sqrt{n\log(n)/p} = \Omega(pn^2)$ and $d_{\max}/d_{\min}= O(1)$.

We notice that, by the bounds on $d_{\min}$ and $d_{\max}$ computed above in \cref{eq: dmax up lw bds,eq: dmin up lw bds},
we have that
\[
    2p(n-2) / (d_{\max}+d_{\min}) = O\left(n^{-1}\right).
\]
Moreover, by assumption $p\ge K\frac{\sqrt[3]{\log n}}{n^{1-(2\gamma/3)}}$, we also get that for a large enough positive constant $K$ and for $\gamma \in (0,1]$,
\[
    C\sqrt{\log (n) / n^3} p^{-3/2} = O\left(n^{-\gamma}\right).
\]
Thus, taking  a large enough positive constant $K$ and for $\gamma \in (0,1]$, we get \begin{equation}\label{eq: erh nu}
    \nu = O\left(n^{-1} + n^{-\gamma}\right)
    = O\left(n^{-\gamma}\right).
\end{equation}

Now we are ready to prove that $\Gamma$ satisfies \cref{assumptions on Gamma} and \cref{extra assumptions on Gamma}.
\begin{itemize}
\item[(\ref{assumptions on Gamma}.1):] 
Recall \cref{eq: erh degree}. Since by assumption $p$ satisfies the condition given in \cref{eq: condition p}, i.e, 
\[
    p\ge K\frac{\sqrt[3]{\log n}}{n^{1-(2\gamma/3)}}\qquad \text{for a large enough } K >0, \quad \gamma \in (0,1],
\]
it immediately follows that $|\nbr{i}| = \Omega(n^{1+(2\gamma/3)}\sqrt[3]{\log n}) = \omega(\log n)$, for every $i\in V$ and every $\gamma\in(0,1]$ with high probability.

\item[(\ref{assumptions on Gamma}.2):] 
From \cref{eq: erh nu}, we have that $\nu = O\left(n^{-\gamma}\right)$, with high probability.
The fact that $\nu<1$ with high probability is a consequence of $\gamma$ being strictly positive.

\item[(\ref{assumptions on Gamma}.3):] 
Recall the definition of $\varepsilon$ given in \cref{def: epsilon delta}.
Note that, by an application of \cref{thm: chernoff bound} on $W_{ij}$ and $D_{ii}$ (see \cref{eq: erh degree,eq: chernoff wij}), it follows that, with high probability, for every $i,j \in V$
\[
    \left| D_{ii} - p(n-2)(n-1) \right| \leq \sqrt{{n \log (n)/p}},
    \qquad
    \left| W_{ij} - p(n-2) \right| \leq \sqrt{{\log (n)/pn}}.
\]
Given the assumptions on $p$ (see \cref{eq: condition p}), we have
\begin{equation}\label{eq: varepsilon erh}
    \varepsilon = \Theta\left(\sqrt{\log (n)/n}\right).
\end{equation}

To validate \cref{assumptions on Gamma}.3, we need to prove that there exists $m\in\mathbb{N}$ such that $\nu^m(1-\nu)^{-1}\sqrt{\Delta n} \le m\varepsilon$ and also $m\varepsilon = o(1)$. From \cref{eq: varepsilon erh}, we see that is enough to take $m = o(\sqrt{n/\log n})$ to have $m\varepsilon = o(1)$.
As for the lower bound on $\varepsilon$, it is verified at least for every $m\in\mathbb{N}$ such that $m \ge \frac{1}{\gamma}$. Indeed, from \cref{eq: erh nu} we find that 
\[
    \nu^m (1-\nu)^{-1} \sqrt{\Delta n} = O\left(n^{\frac{1}{2} -\gamma m}\right),
\]
where we also used that, with high probability, $\Delta = d_{\max}/d_{\min} =  O(1)$ as a consequence of \cref{eq: dmin up lw bds,eq: dmax up lw bds} under the assumption on $p$ (see \cref{eq: condition p}).
Then, from \cref{eq: varepsilon erh}, it follows that every $m > \gamma^{-1}$ satisfies 
\[
    \nu^m (1-\nu)^{-1} \sqrt{\Delta n} \le m\varepsilon.
\]
We proved (\cref{assumptions on Gamma}.3) if $m = o(\sqrt{n/\log n})$ and $m > \gamma^{-1}$, which is true if $\gamma = \Omega(1)$ there exists at least one $m\in\mathbb{N}$ satisfying both the conditions.

\item[(\ref{extra assumptions on Gamma}.1):] 
Recall \cref{eq: varepsilon erh} and let 
\begin{equation}\label{eq: condition delta}
    \delta = o\left( \frac{\sqrt{\sum_{i\in V} D_{ii}^2}}{\sum_{i\in V} D_{ii}} \right).
\end{equation}
Using that with high probability $d_{\max}/d_{\min} = O(1)$, as a consequence of \cref{eq: dmin up lw bds,eq: dmax up lw bds} under the assumption on $p$ (see \cref{eq: condition p}), we have
\begin{gather}\label{eq: upp bd delta}
    \frac{\sqrt{\sum_{i\in V} D_{ii}^2}}{\sum_{i\in V} D_{ii}}
    \leq 
    \frac{\sqrt{\sum_{i\in V} d_{\max}^2}}{\sum_{i\in V} d_{\min}}
    =\frac{1}{\sqrt{n}}\cdot\frac{d_{\max}}{d_{\min}}
    = O\left(\frac{1}{\sqrt{n}}\right),
    \\\label{eq: lw bd delta}
    \frac{\sqrt{\sum_{i\in V} D_{ii}^2}}{\sum_{i\in V} D_{ii}}
    \geq
    \frac{\sqrt{\sum_{i\in V} d_{\min}^2}}{\sum_{i\in V} d_{\max}}
    =\frac{1}{\sqrt{n}}\cdot\frac{d_{\min}}{d_{\max}}
    = \Omega\left(\frac{1}{\sqrt{n}}\right).
\end{gather}
From \cref{eq: upp bd delta}, we get that $\delta = o(1/\sqrt{n})$. As a consequence, from  \cref{eq: varepsilon erh}, we find that $\delta = o(\varepsilon)$.

To conclude the proof of (\ref{extra assumptions on Gamma}.1), we have to prove the existence of $m,\delta$ such that $\delta \gg m\varepsilon^2$, $\delta$ is as in \cref{eq: condition delta}, and $m=o(\sqrt{n/\log n})$ (due to \cref{assumptions on Gamma}.3).
Let us choose, for example, $\delta = {n^{-3/4}}$, which satisfies \cref{eq: condition delta} due to \cref{eq: upp bd delta}, and $m=O(1)$.
Then
\[
m \varepsilon^2 
= O\left( \log (n)/n \right) 
\ll \delta = n^{-3/4}
\ll \varepsilon = O\left(\sqrt{\log (n) / n}\right),
\]
where we used \cref{eq: varepsilon erh} and \cref{eq: lw bd delta}.

\item[(\ref{extra assumptions on Gamma}.2):] 
Recall the bounds on $d_{\min}$ and $d_{\max}$ and that our assumptions on $p$ imply $d_{\max}/d_{\max} < 2$, with high probability. Thus,
\begin{align*}
    \left( \sum_{i\in V} D_{ii}^2 \right)^{-\frac{3}{2}} \!\!\!\sum_{i\in V} D_{ii}^3
    \leq \left( \sum_{i\in V} d_{\min}^2 \right)^{-\frac{3}{2}} \!\!\!\sum_{i\in V} d_{\max}^3
    = \frac{1}{\sqrt{n}}\cdot \frac{d_{\max}^3 }{d_{\min}^3}
    = O\left(\frac{1}{\sqrt{n}}\right) = o(1).
\end{align*}
\end{itemize}
\end{proof}
\section{Simulations}\label{sec: simulations}
We simulate the nonlinear dynamics studied in this paper for $\lambda \in \{-1/4,+1/4\}$, to show two different convergence behaviors of the dynamics. 
The state of each \node $\state{0}_i$ is initially set to $\pm 1$ with probability $1/2$. 
At each time $t>0$, \nodes update their states according to \cref{eq:dynamics clean}, with $s(x)=e^{x}$.

We consider two hypergraph topologies: a Erd\H{o}s-R\'{e}nyi hypergraph (where each triplet of \nodes is connected with some fixed probability) and a 1D torus (where each \node $i$ belongs to hyperedges
$\{i-2, i-1, i\}$, $\{i-1, i, i+1\}$, $\{i, i+1, i+2\}$, with neighbors modulo $n$).
The former satisfies the assumptions of \cref{theorem: main} (see \cref{prop: erh}).
The latter does not satisfy the assumptions of \cref{theorem: main}: \nodes have constant degree 3 irrespective of $n$ hence the hypergraph is not sufficiently dense. 
We set an odd number of \nodes in the 1D torus to prevent time-periodic configurations.

Let $m^{(t)} := \min_{i \in V} \state{t}_i$ and $M^{(t)} := \max_{i \in V} \state{t}_i$.
We stop the simulations at the first time $T$ such that $M^{(T)}-m^{(T)}< 1/n^2$. 
In \cref{fig:simulations}, we plot the evolution over time of $m^{(t)},M^{(t)},\bar{x}^{(t)} \!\!=\! \frac{\sum_{i \in V} D_{ii} \state{t}_i}{\sum_{i \in V} D_{ii}}$ (degree-weighted average of states at time $t$). 
We also plot $\bar{\mu}=\bar{x}^{(0)}$ as a reference line, namely the value the linear dynamics converges to, or equivalently the nonlinear dynamics when $\lambda=0$ (see \cref{prop: linear dynamics convergence,prop: linear interaction}). 

\begin{figure}[htbp]
    \centering
    \vspace{-1em}
    \includegraphics[width=0.95\linewidth]{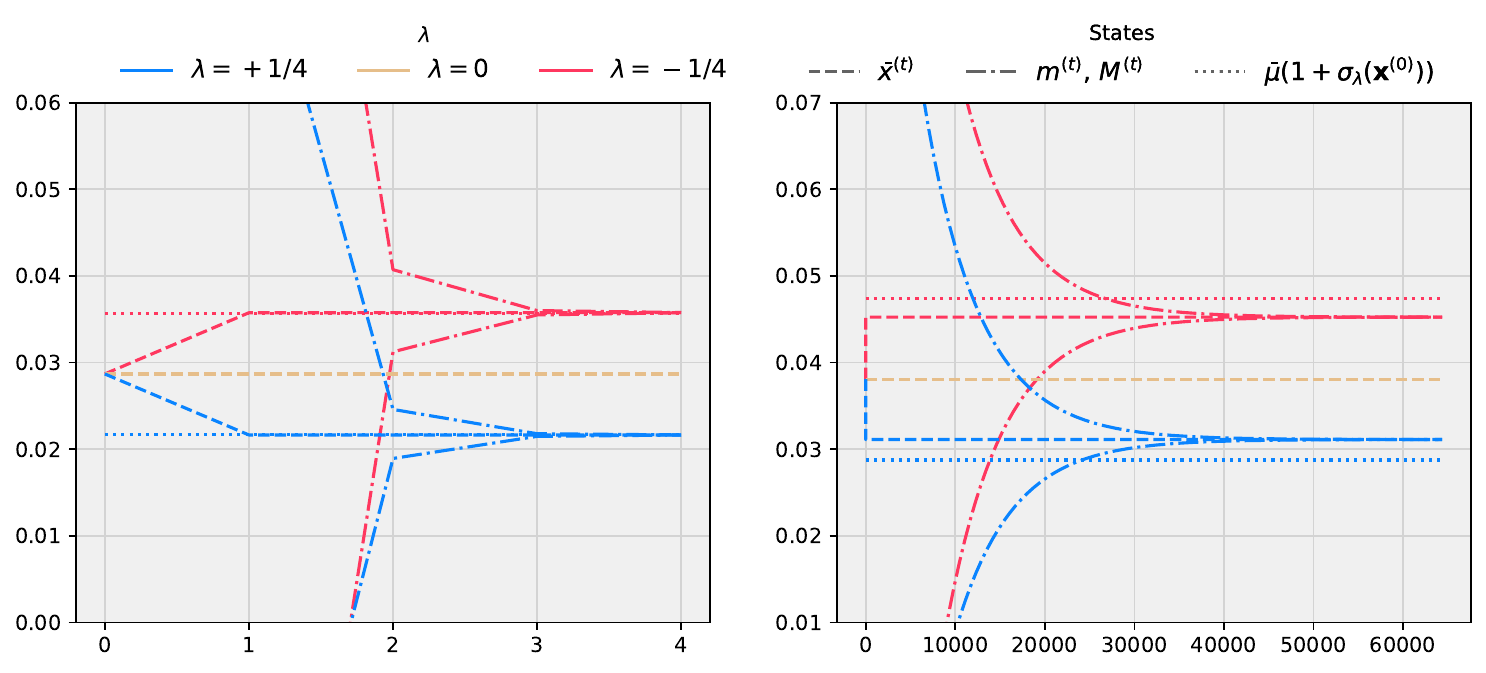}
    \vspace{-1em}
    \caption{Evolution of some significant states of the dynamics ($y$-axis, line styles) over time ($x$-axis), for different values of $\lambda$ (colors) and on different hypergraph topologies (left/right plots).
    Left: Erd\H{o}s-R\'{e}nyi hypergraph ($n=500, p=\log^{1/3}(n)/n$).
    Right: 1D torus hypergraph ($n=499$).}
    \vspace{-0.5em}
    \label{fig:simulations}
\end{figure}

From the simulations we can observe the same behavior described by \cref{theorem: main}, namely a multiplicative shift $(1+\sigma_{\lambda}(\mathbf{x}^{(0)}))$ from $\bar{\mu}$ that is either toward the initial majority when $\lambda<0$ or toward balance (i.e., 0) when $\lambda>0$. 
We measure the convergence error $\xi := |\bar{x}^{(T)}-\bar{\mu}(1+\sigma_{\lambda}(\mathbf{x}^{(0)}))|$ w.r.t.\ our result.
The simulation on the Erd\H{o}s-R\'{e}nyi hypergraph stops at $T=4$ and we have $\xi \approx 4 \cdot 10^{-5}$.
Instead, the simulation on the 1D torus hypergraph stops at $T\approx 65\,000$ and we have $\xi \approx 2 \cdot 10^{-3}$.
The slower convergence is explained by the coarser connectivity of the 1D torus hypergraph, which results in a larger value of $\nu$ (second largest eigenvalue of motif graph), a quantity that governs the convergence time of the dynamics in \cref{theorem: main}.
The larger convergence error $\xi$ instead, seems still within the error bounds given in our theorem since $n=499$. 
In fact, for this hypergraph we have that $\xi \approx 1/n$ and our result essentially claims the convergence error to be $o(|\bar{\mu}|)$ (with $\Ex[|\bar{\mu}|] \approx 1/\sqrt{n}$, given the randomness in the initial data).
Interestingly, the simulations agree with our result even on the 1D torus, which does not satisfy our assumptions. 
This suggests that the result in \cref{theorem: main} might hold for general hypergraph topologies and is worth of future investigation, possibly with higher order corrections beyond $\bar\mu(1+\sigma_\lambda(\mathbf{x}^{(0)}))$
to get a more accurate description of the dynamics.

Moreover, we study how the convergence error $\xi$ varies with respect to the number of nodes $n$.
We consider two hypergraph topologies: Erd\H{o}s-R\'{e}nyi hypergraphs and complete hypergraphs with an odd number of nodes $n\in\{3,\ldots,99\}$.
The convergence error is measured at time $T=4$ (this is large enough for convergence as observed in \cref{fig:simulations}) and averaged over $100$ simulations starting with independent initial states.

\begin{figure}[htbp]
    \centering
    \vspace{-1em}
    \includegraphics[width=0.95\linewidth]{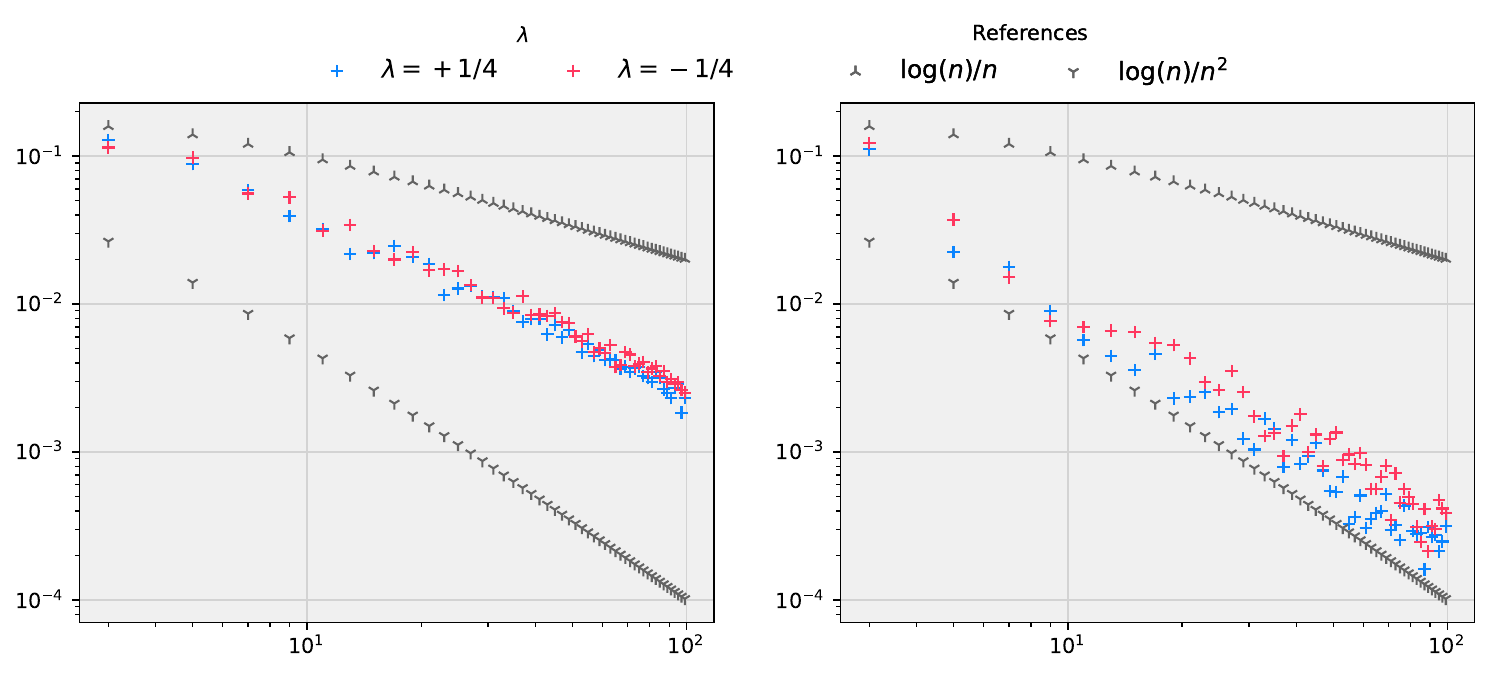}
    \vspace{-1em}
    \caption{Average convergence error $\xi$ of the dynamics ($y$-axis, log-scale) over the number of nodes $n$ ($x$-axis, log-scale), for different values of $\lambda$ (colors) and on different hypergraph topologies (left/right plots). 
    Left: Erd\H{o}s-R\'{e}nyi hypergraph ($p=1/(n\log n)$).
    Right: complete hypergraph.}
    \vspace{-0.5em}
    \label{fig:simulations-error}
\end{figure}

In \cref{fig:simulations-error} we plot the result of our simulations.
We can observe that the error converges to 0 roughly as a polynomial in $n$ (axes are log-scale).
We conjecture that our error estimate\footnote{Note that our error estimate is just an upper bound to the actual error.} is lower bounded by
$\Omega(\log(n) / n^2)$,
as it seems to be for the complete hypergraph. 
As expected, we note that the experimental error is below our estimate for the complete hypergraph of $\Theta(\log(n)/n)$, and that the estimate holds for all tested values of $n$. 
Comparing the error on different topologies, we observe that it is bigger for the sparser Erd\H{o}s-R\'{e}nyi hypergraphs (which are ``farther'' than the complete hypergraph from our mean field model), even if it stays below our estimate for the complete topology.

\bibliographystyle{alpha}
\bibliography{references}

\newcommand{\etalchar}[1]{$^{#1}$}
\begin{thebibliography}{GDPG{\etalchar{+}}21}

\bibitem[AG11]{doi:10.1137/9780898719987}
Uri~M. Ascher and Chen Greif.
\newblock {\em A First Course in Numerical Methods}.
\newblock SIAM, 2011.

\bibitem[BC94]{billick1994higher}
Ian Billick and Ted~J. Case.
\newblock Higher order interactions in ecological communities: What are they and how can they be detected?
\newblock {\em Ecology}, 75(6):1530--1543, 1994.

\bibitem[BCI{\etalchar{+}}20]{BATTISTON20201}
Federico Battiston, Giulia Cencetti, Iacopo Iacopini, Vito Latora, Maxime Lucas, Alice Patania, Jean-Gabriel Young, and Giovanni Petri.
\newblock Networks beyond pairwise interactions: Structure and dynamics.
\newblock {\em Physics Reports}, 874:1--92, 2020.

\bibitem[BCPR20]{BECCHETTI202049}
Luca Becchetti, Emilio Cruciani, Francesco Pasquale, and Sara Rizzo.
\newblock Step-by-step community detection in volume-regular graphs.
\newblock {\em Theoretical Computer Science}, 847:49--67, 2020.

\bibitem[BG79]{brezis1979uniqueness}
Ha{\"\i}m Br{\'e}zis and M.G. GRANDALL.
\newblock Uniqueness of solutions of the initial-value problem for $u_t-{\Delta} \phi(u)= 0$, 1979.

\bibitem[BH11]{brouwer2011spectra}
Andries~E. Brouwer and Willem~H. Haemers.
\newblock {\em Spectra of graphs}.
\newblock Springer, 2011.

\bibitem[BLM{\etalchar{+}}06]{BOCCALETTI2006175}
S.~Boccaletti, V.~Latora, Y.~Moreno, M.~Chavez, and D.-U. Hwang.
\newblock Complex networks: Structure and dynamics.
\newblock {\em Physics Reports}, 424(4):175--308, 2006.

\bibitem[CBCF20]{PhysRevE.101.022308}
Timoteo Carletti, Federico Battiston, Giulia Cencetti, and Duccio Fanelli.
\newblock Random walks on hypergraphs.
\newblock {\em Phys. Rev. E}, 101:022308, Feb 2020.

\bibitem[Che52]{10.2307/2236576}
Herman Chernoff.
\newblock A measure of asymptotic efficiency for tests of a hypothesis based on the sum of observations.
\newblock {\em The Annals of Mathematical Statistics}, 23(4):493--507, 1952.

\bibitem[CR11]{chung2011spectra}
Fan Chung and Mary Radcliffe.
\newblock On the spectra of general random graphs.
\newblock {\em The Electronic Journal of Combinatorics}, pages 215--215, 2011.

\bibitem[DeG74]{degroot1974reaching}
Morris~H. DeGroot.
\newblock Reaching a consensus.
\newblock {\em Journal of the American Statistical Association}, 69(345):118--121, 1974.

\bibitem[DeV21]{deville2021consensus}
Lee DeVille.
\newblock {Consensus on simplicial complexes: Results on stability and synchronization}.
\newblock {\em Chaos: An Interdisciplinary Journal of Nonlinear Science}, 31(2), 02 2021.

\bibitem[Ess42]{esseen1942liapunov}
Carl-Gustav Esseen.
\newblock On the liapunov limit error in the theory of probability.
\newblock {\em Ark. Mat. Astr. Fys.}, 28:1--19, 1942.

\bibitem[FJ90]{friedkin1990social}
Noah~E. Friedkin and Eugene~C. Johnsen.
\newblock Social influence and opinions.
\newblock {\em The Journal of Mathematical Sociology}, 15(3-4):193--206, 1990.

\bibitem[GDPG{\etalchar{+}}21]{gambuzza2021stability}
Lucia~Valentina Gambuzza, Francesca Di~Patti, Luca Gallo, Stefano Lepri, Miguel Romance, Regino Criado, Mattia Frasca, Vito Latora, and Stefano Boccaletti.
\newblock Stability of synchronization in simplicial complexes.
\newblock {\em Nature communications}, 12(1):1255, 2021.

\bibitem[Ham20]{hamilton2020graph}
William~L Hamilton.
\newblock Graph representation learning.
\newblock {\em Synthesis Lectures on Artifical Intelligence and Machine Learning}, 14(3):1--159, 2020.

\bibitem[HK20]{PhysRevE.101.022305}
Leonhard Horstmeyer and Christian Kuehn.
\newblock Adaptive voter model on simplicial complexes.
\newblock {\em Phys. Rev. E}, 101:022305, Feb 2020.

\bibitem[HKB{\etalchar{+}}22]{doi:10.1137/21M1399427}
Abigail Hickok, Yacoub Kureh, Heather~Z. Brooks, Michelle Feng, and Mason~A. Porter.
\newblock A bounded-confidence model of opinion dynamics on hypergraphs.
\newblock {\em SIAM Journal on Applied Dynamical Systems}, 21(1):1--32, 2022.

\bibitem[HNS13]{RevModPhys.85.197}
Hans-Werner Hammer, Andreas Nogga, and Achim Schwenk.
\newblock Colloquium: Three-body forces: From cold atoms to nuclei.
\newblock {\em Rev. Mod. Phys.}, 85:197--217, Jan 2013.

\bibitem[Hoe63]{10.2307/2282952}
Wassily Hoeffding.
\newblock Probability inequalities for sums of bounded random variables.
\newblock {\em Journal of the American Statistical Association}, 58(301):13--30, 1963.

\bibitem[IPBL19]{iacopini2019simplicial}
Iacopo Iacopini, Giovanni Petri, Alain Barrat, and Vito Latora.
\newblock Simplicial models of social contagion.
\newblock {\em Nature communications}, 10(1):2485, 2019.

\bibitem[JMFB15]{doi:10.1137/130913250}
Peng Jia, Anahita MirTabatabaei, Noah~E. Friedkin, and Francesco Bullo.
\newblock Opinion dynamics and the evolution of social power in influence networks.
\newblock {\em SIAM Review}, 57(3):367--397, 2015.

\bibitem[LW20]{lutzeyer2017comparing}
Johannes~F. Lutzeyer and Andrew~T. Walden.
\newblock Comparing spectra of graph shift operator matrices.
\newblock In {\em Complex Networks and Their Applications VIII}, pages 191--202. Springer, 2020.

\bibitem[NL21]{PhysRevE.104.024316}
James Noonan and Renaud Lambiotte.
\newblock Dynamics of majority rule on hypergraphs.
\newblock {\em Phys. Rev. E}, 104:024316, Aug 2021.

\bibitem[NLS21]{PhysRevE.104.064305}
Leonie Neuh\"auser, Renaud Lambiotte, and Michael~T. Schaub.
\newblock Consensus dynamics on temporal hypergraphs.
\newblock {\em Phys. Rev. E}, 104:064305, Dec 2021.

\bibitem[NML20]{PhysRevE.101.032310}
Leonie Neuh\"auser, Andrew Mellor, and Renaud Lambiotte.
\newblock Multibody interactions and nonlinear consensus dynamics on networked systems.
\newblock {\em Phys. Rev. E}, 101:032310, Mar 2020.

\bibitem[PBT22]{pmlr-v162-prokopchik22a}
Konstantin Prokopchik, Austin~R Benson, and Francesco Tudisco.
\newblock Nonlinear feature diffusion on hypergraphs.
\newblock In {\em Proc. of the 39th International Conference on Machine Learning}, pages 17945--17958, 2022.

\bibitem[PPV17]{patania2017shape}
Alice Patania, Giovanni Petri, and Francesco Vaccarino.
\newblock The shape of collaborations.
\newblock {\em EPJ Data Science}, 6:1--16, 2017.

\bibitem[PT17]{PROSKURNIKOV201765}
Anton~V. Proskurnikov and Roberto Tempo.
\newblock A tutorial on modeling and analysis of dynamic social networks. part i.
\newblock {\em Annual Reviews in Control}, 43:65--79, 2017.

\bibitem[PT18]{PROSKURNIKOV2018166}
Anton~V. Proskurnikov and Roberto Tempo.
\newblock A tutorial on modeling and analysis of dynamic social networks. part ii.
\newblock {\em Annual Reviews in Control}, 45:166--190, 2018.

\bibitem[PTVF07]{press2007numerical}
William~H. Press, Saul~A. Teukolsky, William~T. Vetterling, and Brian~P. Flannery.
\newblock {\em Numerical Recipes: The Art of Scientific Computing}.
\newblock Cambridge University Press, 3rd edition, 2007.

\bibitem[SBSB06]{schneidman2006weak}
Elad Schneidman, Michael~J Berry, Ronen Segev, and William Bialek.
\newblock Weak pairwise correlations imply strongly correlated network states in a neural population.
\newblock {\em Nature}, 440(7087):1007--1012, 2006.

\bibitem[Sha22]{shang2022system}
Yilun Shang.
\newblock A system model of three-body interactions in complex networks: Consensus and conservation.
\newblock {\em Proceedings of the Royal Society A}, 478(2258):20210564, 2022.

\bibitem[Sha23]{10.1093/comnet/cnad009}
Yilun Shang.
\newblock {Non-linear consensus dynamics on temporal hypergraphs with random noisy higher-order interactions}.
\newblock {\em Journal of Complex Networks}, 11(2), 03 2023.

\bibitem[She10]{shevtsova2010improvement}
Irina~G. Shevtsova.
\newblock An improvement of convergence rate estimates in the lyapunov theorem.
\newblock {\em Doklady Mathematics}, 82(3):862--864, 2010.

\bibitem[SNL21]{Sahasrabuddhe_2021}
Rohit Sahasrabuddhe, Leonie Neuhäuser, and Renaud Lambiotte.
\newblock Modelling non-linear consensus dynamics on hypergraphs.
\newblock {\em Journal of Physics: Complexity}, 2(2):025006, feb 2021.

\bibitem[SOB{\etalchar{+}}16]{schaub2016graph}
Michael~T. Schaub, Neave O'Clery, Yazan~N. Billeh, Jean-Charles Delvenne, Renaud Lambiotte, and Mauricio Barahona.
\newblock {Graph partitions and cluster synchronization in networks of oscillators}.
\newblock {\em Chaos: An Interdisciplinary Journal of Nonlinear Science}, 26(9), 08 2016.

\bibitem[SS90]{stewart1990matrix}
Gilbert~W. Stewart and Ji-guang Sun.
\newblock {\em Matrix perturbation theory}.
\newblock Academic press, 1990.

\bibitem[SSL16]{sekara2016fundamental}
Vedran Sekara, Arkadiusz Stopczynski, and Sune Lehmann.
\newblock Fundamental structures of dynamic social networks.
\newblock {\em Proceedings of the national academy of sciences}, 113(36):9977--9982, 2016.

\bibitem[TBP21]{10.1145/3442381.3450035}
Francesco Tudisco, Austin~R. Benson, and Konstantin Prokopchik.
\newblock Nonlinear higher-order label spreading.
\newblock In {\em Proc. of the Web Conference 2021}, WWW '21, page 2402–2413. Association for Computing Machinery, 2021.

\bibitem[vdSRJ16]{doi:10.1080/00207179.2015.1095353}
A.~J. van~der Schaft, S.~Rao, and B.~Jayawardhana.
\newblock A network dynamics approach to chemical reaction networks.
\newblock {\em International Journal of Control}, 89(4):731--745, 2016.

\bibitem[VH69]{Volpert_1969}
A~I Vol'pert and S~I Hudjaev.
\newblock Cauchy's problem for degenerate second order quasilinear parabolic equations.
\newblock {\em Mathematics of the USSR-Sbornik}, 7(3):365, apr 1969.

\bibitem[WGA21]{Jennifer2021}
Jennifer Weissen, Simone Göttlich, and Dieter Armbruster.
\newblock Density dependent diffusion models for the interaction of particle ensembles with boundaries.
\newblock {\em Kinetic and Related Models}, 14(4):681--704, 2021.

\end{thebibliography}

\end{document}